\documentclass[11pt]{article}
\newcommand{\filename}{Gen-DN-ArXiv-2017-08-20v3.tex}
\newcommand{\comment}[1]{}
\usepackage{makeidx}         
\usepackage{graphicx}        
\usepackage{psfrag}          
\usepackage{multicol}        
\usepackage{amsmath}
\usepackage{amssymb}
\usepackage{amsfonts}
\usepackage{amsthm}
\usepackage{latexsym}
\usepackage[T1]{fontenc}
\usepackage{subfigure}                      
\usepackage{float}
\usepackage{mathrsfs}
\usepackage{color}
\newcommand{\rd}{\color{red}}

\allowdisplaybreaks
\newtheorem{lemma}{LEMMA}[section]
\newtheorem{theorem}[lemma]{THEOREM}
\newtheorem{definition}[lemma]{DEFINITION}
\newtheorem{corollary}[lemma]{COROLLARY}

\newtheorem{remark}[lemma]{REMARK}

\def\supp{\mathop{\mbox{\rm supp}}\nolimits}
 \newcommand{\nc}{\newcommand}
 \nc{\ha}{\frac{1}{2}}
 \nc{\tha}{\frac{3}{2}}
 \nc{\s}{\widetilde}
 \nc{\dst}{\displaystyle}
 \nc{\gm}{\gamma}
 \nc{\ga}{\Gamma}
 \nc{\ka}{\kappa}
 \nc{\eps}{\varepsilon}
 \nc{\vep}{\varepsilon}
 \nc{\hi}{\varphi}
 \nc{\vfi}{\varphi}
 \nc{\oa}{\Omega}
 \nc{\Om}{\Omega}
 \nc{\om}{\omega}
 \nc{\ov}{\overline}
 \nc{\lon}{\longrightarrow}
 \nc{\scr}{\scriptstyle}
 \nc{\ex}{\exists}
 \nc{\fo}{\forall}
 \nc{\pa}{\partial}
 \nc{\pO}{{\partial\Omega}}
 \nc{\und}{\underline}
 \nc{\ze}{\zeta}
 \nc{\si}{\sigma}
 \nc{\tri}{\triangle}
 \nc{\al}{\alpha}
 \nc{\bt}{\beta}
 \nc{\lf}{\left}
 \nc{\ri}{\right}
 \nc{\lm}{\lambda}
 \nc{\lam}{\lambda}
 \nc{\dt}{\delta}
 \nc{\de}{\delta}
 \nc{\te}{\theta}
 \nc{\tl}{\tilde}
 \nc{\wt}{\widetilde}
 \nc{\p}{\prime}
 \nc{\m}{\mu}
 \nc{\R}{{\mathbb R}}
 \nc{\B}{{\Bbb B}}
 \nc{\N}{{\Bbb N}}
 \nc{\C}{{\Bbb C}}

\newcommand{\E}{{\cal E}}
\newcommand{\F}{{\cal F}}

\renewcommand{\L}{{\cal L}}
\renewcommand{\P}{{\cal P}}

\newcommand{\U}{{\cal U}}

\newcommand{\V}{{\cal V}}

\newcommand{\W}{{\cal W}}
\newcommand{\Wp}{{\W\,^\prime}}
\newcommand{\be}{\begin{equation}}
\newcommand{\ee}{\end{equation}}
\newcommand{\bes}{\begin{equation*}}
\newcommand{\ees}{\end{equation*}}
\newcommand{\bea}{\begin{eqnarray}}
\newcommand{\eea}{\end{eqnarray}}
\newcommand{\beas}{\begin{eqnarray*}}
\newcommand{\eeas}{\end{eqnarray*}}

\def\text#1{\;\;\mbox{#1}\;}

\vspace{1cm}
 \makeatletter
\renewcommand{\@oddhead}{\vbox{\hbox to\textwidth{\scriptsize %
 \filename\hfill S.E.Mikhailov\hfill 
 \hfill\arabic{page}}
 \vspace{0.25ex}
 
 \hrule
 }}
 \renewcommand{\@evenhead}{\vbox{\hbox to\textwidth{\scriptsize %
 \filename\hfill 
 }
  \vspace{0.25ex}
  
\hrule
}}
 \makeatother
 \allowdisplaybreaks
 \numberwithin{equation}{section}
\begin{document}
\title
{
Analysis of Segregated
Boundary-Domain Integral Equations for
 Variable-Coefficient Dirichlet and Neumann Problems\\ 
 with General Data
}

\author
{S.E. Mikhailov
\footnote{%
e-mail: {\sf sergey.mikhailov@brunel.ac.uk}}\\ 
Department of Mathematics, Brunel University London, UK
     \\
}
 \date{\em \today
 \hfill \filename }

 \maketitle

\maketitle

\begin{abstract}\noindent

Segregated direct
boundary-domain integral equations (BDIEs) based on a  parametrix and
associated with the Dirichlet and Neumann boundary value problems for the linear stationary diffusion
partial differential equation with a variable coefficient are formulated. 
The PDE right hand sides belong to the Sobolev space $H^{-1}(\Omega)$ or $\widetilde H^{-1}(\Omega)$, when neither classical nor canonical co-normal derivatives are well defined.
Equivalence of the BDIEs to the original BVP, BDIE solvability, solution uniqueness/non-uniqueness,
and as well as Fredholm property and invertibility of the BDIE operators  are
analysed in Sobolev (Bessel potential) spaces.
It is shown that the BDIE operators for the Neumann BVP are not invertible, and
appropriate finite-dimensional perturbations are constructed leading to invertibility of the perturbed operators.\\


\end{abstract}

\section{Introduction}

Many applications in science and engineering can be modeled by
boundary-value problems (BVPs) for partial differential equations with variable coefficients.
Reduction of the BVPs with arbitrarily variable coefficients to explicit
boundary integral equations is usually not possible, since the
fundamental solution necessary for such reduction is generally not
available in an analytical form (except for some special
dependence of the coefficients on coordinates). Using a parametrix
(Levi function) introduced in \cite{Levi1909}, \cite{Hilbert1912}
as a substitute of a fundamental solution, it is possible however
to reduce such a BVP to a system of boundary-domain integral equations, BDIEs, (see
e.g.  \cite[Sect. 18]{Miranda1970}, \cite{Pomp1998a, Pomp1998b}, where the
Dirichlet, Neumann and Robin problems for some PDEs were reduced
to {\em indirect} BDIEs). However, many questions about their
equivalence to the original BVP, solvability, solution uniqueness
and invertibility of corresponding integral operator remained
open for rather long time.

In \cite{CMN-1, CMN-2, MikMMAS2006, CMN-NMPDE-crack, CMN-Ext-AA2013}, the 3D mixed (Dirichlet-Neumann)
boundary value problem (BVP) for the variable-coefficient
stationary diffusion PDE {\em with a square integrable right hand side} was
considered. 
Such equations appear e.g. in electrostatics,
stationary heat transfer and other diffusion problems for
inhomogeneous media. The BVP has been reduced to either segregated
or united direct Boundary-Domain Integral or Integro-Differential
Equations,  some of the which are associated with
those formulated in \cite{MikEABEM2002}. 

For a function from the Sobolev space $H^1(\Omega)$,  a classical co-normal derivative
in the sense of traces may not exist. 
However, when this function satisfies a second order partial differential equation with a right-hand side from $H^{-1}(\Omega)$,  the generalised co-normal derivative can be defined in the weak sense, associated with the first Green identity and an extension of the PDE right hand side to $\widetilde H^{-1}(\Omega)$ (see \cite[Lemma 4.3]{McLean2000}, \cite[Definition 3.1]{MikJMAA2011}).
Since the extension is non-unique, the co-normal derivative appears to be a non-unique operator,
which is also non-linear in $u$ unless a linear relation between
$u$ and the PDE right hand side extension is enforced. This
creates some difficulties in formulating the
boundary-domain integral equations. 

These difficulties are addressed in this paper
presenting formulation and analysis of direct
segregated BDIE systems equivalent to the Dirichlet and Neumann boundary value
problems for the divergent-type PDE with a variable scalar coefficient and a general
right hand side from ${H}^{-1}(\Omega)$ extended when necessary to  $\widetilde{H}^{-1}(\Omega)$. 
This needed a non-trivial generalisation of the third Green identity and its co-normal derivative for such functions,
which essentially extends the approach implemented in \cite{CMN-1, CMN-2, MikMMAS2006, CMN-NMPDE-crack, CMN-Ext-AA2013} for the right hand side from $L_2(\Omega)$. 
Equivalence of the BDIEs to the original BVP, BDIE solvability, solution uniqueness/non-uniqueness,
as well as Fredholm property and invertibility of the BDIE operators  are
analysed in Sobolev (Bessel potential) spaces.
It is shown that the BDIE operators for the Neumann BVP are not invertible, and
appropriate finite-dimensional perturbations are constructed leading to invertibility of the perturbed operators.

Note that our analysis is aimed not at the boundary-value problems, which properties are well-known nowadays, but rather at the BDIE systems per se. The analysis is interesting not only in its own rights but is also to be used further on  for analysis of convergence and stability of BDIE-based numerical methods for PDEs,
 see e.g. \cite{GMR2013, MikEABEM2002, MikNakJEM, MikMoh2012, SSA2000, SSZ2005, Taigbenu1999, ZZA1998, ZZA1999}.

\section{Co-normal derivatives and boundary value problems}
Let $\Omega$ be a bounded open three--dimensional region of
$\R^3$. For simplicity, we assume that the boundary $\pa \Omega$
is a simply connected, closed, infinitely smooth surface. Let
$a\in C^{\infty}(\overline{\Omega})$, $a(x)>0$ for $x\in
\overline{\Omega}$. Let also $\pa_{x_j}:=\pa/\pa{x_j}$
 $(j=1,2,3)$, $\pa_x:=\nabla_x=(\pa_{x_1},\pa_{x_2}, \pa_{x_3})$.

We consider the scalar elliptic differential equation, which for
sufficiently smooth $u$ has the following strong form,
\begin{equation}
\label{2.1} \dst Au(x):=A(x,\pa_x)\,u(x) := \sum\limits_{i=1}^3
\frac{\pa}{\pa x_i}\, \Big( \, a(x)\,\frac{\pa u(x)}{\pa x_i}\,
\Big) = f(x),  \;\;\;\; x \in \Omega,
\end{equation}
where
 $u$ is an unknown
function and $f$ is a given function in $\Omega$.

 In what follows $\mathcal D(\Omega)=C^\infty_{comp}(\Omega)$, $ H^s(\Omega)= H^s_2(\Omega)$, $ H^s(\pO)=H^s_2(\pO)$ are the
Bessel potential spaces, where $s\in \R$ is an arbitrary real
number (see, e.g., \cite{LiMa1}, \cite{McLean2000}). We recall that $H^s$
coincide with the Sobolev--Slobodetski spaces $W^s_2$ for any
non-negative $s$.
We denote by $\s{H}^s (\Omega)$ the subspace of ${H}^s (\R^3)$,
$$
\s{H}^s (\Omega):=\{g:\;g\in H^s  (\R^3),\; \supp \,g
\subset\ov{\Omega}\},
$$
while ${H}^s (\Omega)$ denotes the space of restrictions on
$\Omega$ of distributions  from ${H}^s  (\R^3)$,
$$
{H}^s (\Omega):=\{r_{_{\Omega}}g:\;g\in{H}^s (\R^3)\},
$$
where $r_{_{\Omega}}$ denotes the restriction operator on
$\Omega$. We will also use notation
$g|_{_{\Omega}}:=r_{_{\Omega}}g$.
We denote by ${H}^s_{\pO}$ the following subspace of ${H}^s
(\R^3)$ (and $\s{H}^s (\Omega)$),
\be\label{H_dO}
{H}^s_{\pO}:=\{g:\;g\in H^s  (\R^3),\; \supp \,g \subset{\pO}\}.
\ee


From the trace theorem (see e.g. \cite{LiMa1, DaLi4, McLean2000})
for $u\in H^1(\Omega )$, it follows that
 $\gamma^+\,u \in H^{\ha}(\pO)$, where
$\gamma^+=\gamma^+_{_\pO}$ are the trace operators on $\pO$ from $\Omega $.
Let also $\gamma^{-1}:H^{1/2}(\partial\Omega)\to H^1(\Omega)$ denote a (non-unique) continuous right inverse to the trace operators $\gamma^+$, i.e., $\gamma^+_{\partial\Omega}\gamma^{-1}w=w$ for any $w\in H^{1/2}(\partial\Omega)$, and $(\gamma^{-1})^*:\s{H}^{-1}(\Omega)\to H^{-\ha}(\pO)$ is the continuous operator dual to $\gamma^{-1}:H^{1/2}(\partial\Omega)\to H^1(\Omega)$, i.e., 
$\langle(\gamma^{-1})^*\tilde f,w\rangle_{\Omega}:=\langle\tilde f,\gamma^{-1}w\rangle_{\Omega}$ for any
$\tilde f\in \s{H}^{-1}(\Omega)$ and $w\in H^{1/2}(\partial\Omega)$.

For $u\in H^2(\Omega)$  we can denote by $T^+$ the corresponding classical (strong)
co-normal derivative operator on $\pO$ in the sense of
traces,
\begin{align}\label{Tcl}
\dst T^+u(x) := \sum\limits_{i=1}^3
a(x)\,n_i(x)\gamma^+\frac{\pa u(x)}{\pa x_i}
=a(x)\,\gamma^+\frac{\pa u(x)}{\pa n(x)},
\end{align}
where $n^+(x)$ is the outward (to $\Omega$) unit normal vectors at
the point $x\in \pO$. However the classical co-normal derivative operator is, generally, not well defined if $u\in H^1(\Omega)$ (cf. an example in Section~\ref{Example} in Appendix).

For $u\in H^1(\Omega)$, 
the partial differential operator $A$ is understood in the sense of distributions,
\begin{equation}\label{Ldist}
    \langle Au,v \rangle_\Omega:=-\E(u,v)\quad \forall v\in
    \mathcal D(\Omega),
\end{equation}
where 
 $$
\E(u,v):=\int_{\Omega} a(x) \;\nabla u(x) \cdot\nabla v(x)dx,
 $$
and the duality brackets $\langle \;g,\;\cdot\;\rangle_\Omega $ denote
value of a linear functional (distribution) $g$, extending the
usual $L_2$ dual product.

Since the set $\mathcal D(\Omega)$ is dense in $\s{H}^1
(\Omega)$, the above formula defines a continuous operator $A:
H^1(\Omega)\to H^{-1}(\Omega)=[\s{H}^1 (\Omega)]^*$,
\begin{equation}\label{LH1}
    \langle Au,v \rangle_\Omega:=-\E(u,v),\quad \forall\ u\in H^1(\Omega),\, v\in\s{H}^1(\Omega).
\end{equation}
Let us consider also the different operator, $\check{A}:
H^1(\Omega)\to \s{H}^{-1}(\Omega)=[{H}^1 (\Omega)]^*$,
\begin{multline}\label{Ltil}
    \langle \check{A}u,v \rangle_\Omega:=-\E(u,v)
    =-\int_{\Omega} a(x)\nabla u(x)\cdot\nabla v(x)dx
        =-\int_{\R^3} \mathring E [a\nabla u](x)\cdot\nabla V(x)dx\\
        =\langle \nabla\cdot\mathring E [a\nabla u],V\rangle_{\R^3}
        =\langle \nabla\cdot\mathring E [a\nabla u],v\rangle_\Omega, \quad \forall\, u\in H^1(\Omega),\  v\in
    {H}^1(\Omega),
\end{multline}
which is evidently continuous and can be written as 
\begin{align}\label{checkA}
\check{A}u:=\nabla\cdot\mathring E [a\nabla u].
\end{align}
Here $V\in H^1(\R^3)$ is such that $r_\Omega V=v$ and $\mathring E$ denotes the operator of extension of the functions, defined in $\Omega$, by zero outside $\Omega$ in $\R^3$. 
For any $u\in H^1(\Omega)$, the
functional $\check{A}u$ belongs to $\s{H}^{-1}(\Omega)$ and is an
extension of the functional ${A}u\in {H}^{-1}(\Omega)$, which domain is thus extended from
$\s{H}^1(\Omega)$ to the domain ${H}^1(\Omega)$ for $\check{A}u$.

\comment{The extension is not unique, and any functional of the form
\begin{equation}\label{Ext}
     \check{A}u +g,\quad  g\in
    {H}^{-1}_{\partial\Omega}
\end{equation}
provides another extension. On the other hand, any extension of
${A}u$ from $\s{H}^1(\Omega)$ to ${H}^1(\Omega)$ has evidently
form \eqref{Ext}.
}

Inspired by the first Green identity for smooth functions, we can define {\em the
generalised co--normal derivative} (cf., for
example, \cite[Lemma 4.3]{McLean2000}), \cite[Definition 3.1]{MikJMAA2011}, \cite[Lemma 2.2]{KLW2015ZAMP}).
\begin{definition}\label{GCDd}
Let $u\in {H}^{1}(\Omega)$ and $Au=r_\Omega\tilde f$ in $\Omega$ for some $\tilde f\in\s{H}^{-1}(\Omega)$. Then the generalised co--normal derivative $T^+(\tilde f,u) \in H^{-\ha}(\pO)$  is defined as
\begin{align}
\label{Tgend} 
\left\langle T^+(\tilde f,u)\,,\, w\right\rangle _{\pO}:=
 \langle \tilde f,\gamma^{-1}w \rangle_\Omega + \E(u,\gamma^{-1}w)
 =\langle\tilde f-\check{A}u,\gamma^{-1}w\rangle_\Omega,
  \quad  
\forall\ w\in H^{1/2} (\partial\Omega),
\end{align}
that is,
$
 T^+(\tilde f,u):=(\gamma^{-1})^*(\tilde f-\check{A}u).
$
\end{definition}
By \cite[Lemma 4.3]{McLean2000}), \cite[Theorem 5.3]{MikJMAA2011}, 
we have the estimate
\begin{equation}\label{estimate}
\|T^+(\tilde{f},u)\|_{H^{-1/2}(\pO)}\le
C_1\|u\|_{H^1(\Omega)} + C_2\|\tilde{f}\|_{\s{H}^{-1}(\Omega)},
\end{equation}
and the first Green identity holds in the following form for $u\in {H}^{1}(\Omega)$ such that $Au=r_\Omega\tilde f$ in $\Omega$ for some 
$\tilde f\in\s{H}^{-1}(\Omega)$,
\begin{equation}
\label{Tgen} 
\left\langle T^+(\tilde f,u)\,,\, \gamma^{+}v \right\rangle _{\pO}
=\langle \tilde f,v \rangle_\Omega + \E(u,v)
=\langle\tilde f-\check{A}u,v \rangle_\Omega \quad  \forall\ v\in H^1 (\Omega).
\end{equation}
As follows from Definition~\ref{GCDd}, the generalized co-normal derivative is nonlinear with respect to $u$ for a fixed $\tilde{f}$, but still linear with respect to the couple $(\tilde f, u)$, i.e.,
\begin{align}\label{GCDL}
\alpha_1 T^+(\tilde{f}_1,u_1) +\alpha_2 T^+(\tilde{f}_2,u_2)=
T^+(\alpha_1\tilde{f}_1,\alpha_1u_1) + T^+(\alpha_2\tilde{f}_2,\alpha_2u_2)=
T^+(\alpha_1\tilde{f}_1+\alpha_2\tilde{f}_2,\alpha_1u_1+\alpha_2u_2)
\end{align}
for any complex numbers $\alpha_1,\alpha_2$.

Let us also define some subspaces of $H^s(\Omega)$, cf. \cite{Grisvard1985, Costabel1988, MikJMAA2011, MikJMAA2013}.
\begin{definition}\label{Hst}
Let $s\in\mathbb{R}$ and $A_*:H^s(\Omega)\to {\cal D}^*(\Omega)$
be a linear operator. For $t\ge -\frac{1}{2}$, we introduce the space
 $$
 H^{s,t} (\Omega;A_*):=\{g:\;g\in H^s (\Omega),\
 A_*g|_{\Omega}=\tilde{f}_g|_{\Omega},\ \tilde{f}_g\in
\widetilde{H}^{t}(\Omega)\}
 $$
endowed with the norm
 $ 
 \|g\|_{ H^{s,t} (\Omega;A_*)}:= \left(\|g\|^2_{ H^s(\Omega)}+\|\tilde{f}_g\|^2_{\widetilde{H}^{t}(\Omega)}\right)^{1/2}
 $
 and 
 the inner product 
 \be\label{HstIP}
 (g,h)_{H^{s,t}(\Omega,A_*)}:=(g,h)_{H^{s}(\Omega)} + (\tilde{f}_g,\tilde{f}_h)_{\widetilde H^{t}(\Omega)}.
 \ee
 \end{definition}

The distribution  $\tilde{f}_g\in \widetilde{H}^{t}(\Omega)$, $t\ge
-\frac{1}{2}$, in the above definition is an extension of the distribution
$A_*g|_{\Omega}\in {H}^{t}(\Omega)$, and the extension  is
unique (if it does exist) since any distribution from the space
$H^t(\mathbb{R}^3)$ with a support in $\partial \Omega$ is identical zero if
$t\geq -1/2$ (see e.g. \cite[Lemma 3.39]{McLean2000}, \cite[Theorem 2.10]{MikJMAA2011}). 
We denote this extension as the operator $\tilde A_*$, i.e., $\tilde A_*g=\tilde f_g$.
The uniqueness implies that the norm $
\|g\|_{ H^{s,t} (\Omega;A_*)}$ is well defined. 

We will mostly use the operators $A$ or $\Delta$ as $A_*$ in the above definition. Note that since
$
Au -a\Delta u =\nabla a\cdot\nabla u 
\in L_2(\Omega)
$
 for $u\in H^1(\Omega)$,  we have
 $H^{1,0}(\Omega;A)= H^{1,0} (\Omega;\Delta)$.

\begin{definition}\label{Dccd}
For $u\in H^{1,-\frac{1}{2}}(\Omega;A)$,   we define the
{\em canonical  co-normal derivative} $T^+u \in
H^{\frac{1}{2}}(\partial\Omega)$ as
\begin{equation}\label{Tcandef}
 \left\langle  T^+ u\,,\, w \right\rangle _{\partial\Omega}
 :=\langle \tilde{A} u,\gamma^{-1}w \rangle_{\Omega} +\E(u,\gamma^{-1}w)
=\langle\tilde Au-\check{A}u,v \rangle_\Omega 
  \quad
\forall\ w\in H^{\frac{1}{2}} (\partial\Omega),
\end{equation}
that is,
$
 T^+u:=(\gamma^{-1})^*(\tilde Au-\check{A}u).
$
\end{definition}

The canonical co-normal derivatives $T^+ u$ is
independent of (non-unique) choice of the operator $\gamma^{-1}$, the operator $T^+ :
H^{1,-\frac{1}{2}}(\Omega;A)\to H^{-\frac{1}{2}}(\partial\Omega)$ is continuous, and
the first Green identity holds in the following form,
\begin{eqnarray} \label{Tcan}
 \left\langle T^+u\,,\gamma^+ v\right\rangle _{\partial\Omega}
 =\langle \tilde{A} u,v \rangle_{\Omega}\ +\ \E(u,v)
  \quad
\forall\ v\in H^1 (\Omega).
\end{eqnarray}
The operator
$T^+: H^{1,t}(\Omega;A)\to H^{-\frac{1}{2}}(\partial\Omega)$ in Definition~\ref{Dccd} is continuous for any $t\ge
-\frac{1}{2}$.
The canonical co-normal derivative is defined by the function $u$
and operator $A$ only and does not depend separately on the right hand side
$\tilde f$ (i.e. its behaviour on the boundary), unlike the generalised
co-normal derivative defined in \eqref{Tgen}, and the operator $T^+$ is linear.
Note that the canonical co-normal derivative coincides with the
classical co-normal derivative $T^+u=a\frac{\partial u}{\partial
n}$ if the latter does exist in the trace sense, see \cite[Corollary 3.14 and Theorem 3.16]{MikJMAA2011}.

Let  $u\in H^{1,-\ha}(\Omega;A)$. Then Definitions \ref{GCDd} and \ref{Dccd} imply
that the generalised co-normal derivative for arbitrary extension
$\tilde f\in \s{H}^{-1}(\Omega)$ of the distribution $Au$ can be
expressed as
\begin{equation}
\label{Tgentil} \left\langle
 T^+(\tilde f,u)\,,\, w
\right\rangle _{\pO}=\left\langle T^{+}u\,,\, w
\right\rangle _{\pO}+\langle\tilde f-\tilde A u,\gamma^{-1}w \rangle_\Omega
  \quad  
\forall\ w\in H^{\frac{1}{2}} (\Omega)
 .
\end{equation}

Let $u\in H^1(\Omega)$ and $v\in H^{1,0}(\Omega;A)$. 

Swapping over the roles of $u$ and $v$ in \eqref{Tcan},
we obtain the first Green identity for $v$,
\begin{equation} \label{Greentilde}
\E(u,v)
 +\int_{\Omega}\, u(x)Av(x)dx=
 \left\langle T^{+}v\,,\, \gamma^+u\right\rangle_{\pO}
 .
\end{equation}
If, in addition, $Au=\tilde f$ in $\Omega$, where $\tilde f\in
\s{H}^{-1}(\Omega)$, then according to the definition of
$T^+(\tilde f,u)$ in \eqref{Tgen}, the second Green identity can be
written as
\begin{equation} \label{2.5s}
\langle\tilde f,v \rangle_\Omega\
 -\int_{\Omega}\, u(x)Av(x)dx=
 \left\langle T^+(\tilde f,u)\,,\, \gamma^+ v\right\rangle _{\pO}
 -\left\langle T^{+}v\,,\, \gamma^+u\right\rangle _{\pO}
 .
\end{equation}

If, moreover, $u,v\in H^{1,0}(\Omega;A)$, then we arrive at
the familiar form of the second Green identity for the canonical
extension and canonical co-normal derivatives
\begin{equation} \label{GreenCan}
\int_{\Omega}[v(x)Au(x)- u(x)Av(x)]dx=
 \left\langle T^{+}u\,,\, \gamma^+ v\right\rangle _{\pO}
 -\left\langle T^{+}v\,,\, \gamma^+u\right\rangle _{\pO}
 .
\end{equation}


\section{Parametrix and potential type operators}


We will say, a function $P(x,y)$ of two variables $x,y\in \Omega$
is a parametrix (the Levi function)  for the operator $A(x,\pa_x)$
in $\R^3$ if (see, e.g., \cite{Levi1909, Hilbert1912, Miranda1970, Hellwig1977,
Pomp1998a, Pomp1998b, MikEABEM2002})
\begin{align}
\label{3.1} \dst A(x,\pa_x)\,P(x,y)=\delta (x-y) +R(x,y),
\end{align}
where $\delta(\cdot)$ is the Dirac distribution and $R(x,y)$
possesses a weak (integrable) singularity at $x=y$, i.e.,
\begin{equation}
\label{3.2} \dst R(x,y)={\cal
O}\,(|x-y|^{-\varkappa})\;\;\;\mbox{\rm with}\;\;\;\; \varkappa<3.
\end{equation}
It  is easy to see that for the operator $A(x,\pa_x)$ given by the
left-hand side in \eqref{2.1}, the function
\begin{equation}
\label{3.3} P(x,y)=\frac{1}{a(y)}P_\Delta(x,y)=\frac{-1}{4\pi
\,a(y)\,|x-y|}\,,\;\;\;x,y\in \R^3,
\end{equation}
is a parametrix, while the corresponding remainder function is
\begin{equation}
\label{3.4} R(x,y)=\nabla a(x)\cdot\nabla_x P(x,y)=-\frac{1}{a(y)}\nabla a(x)\cdot\nabla_yP_\Delta(x,y)
=\frac{(x-y)\cdot\nabla a(x)}{4\pi\,a(y)\,|x-y|^3} \,,\;\;\;x,y\in \R^3,
\end{equation}
and satisfies estimate \eqref{3.2} with $\varkappa=2$, due to the
smoothness of the function $a(x)$.
Here \begin{equation}
\label{3.3D} P_\Delta(x,y)=\frac{-1}{4\pi\,|x-y|}\,,\;\;\;x,y\in \R^n
\end{equation}
is the fundamental solution of the Laplace equation.
Evidently, the parametrix $P(x,y)$ given by \eqref{3.3} is related with the
fundamental solution to the operator 
$A(y,\pa_x):=a(y)\Delta(\pa_x)$ 
with "frozen" coefficient $a(x)=a(y)$ and
 $
  A(y,\pa_x)\,P(x,y)=\delta (x-y).
 $


Let $a\in C^\infty(\R^3)$ and $a>0$ a.e. in $\R^3$. For scalar functions $g$, for which the integrals have sense,
the parametrix-based  volume potential operator and the remainder
potential operator, corresponding to parametrix \eqref{3.3} and to
remainder \eqref{3.4} are defined as
\begin{eqnarray}
&&{\bf P}g(y):= \int\limits_{\R^3} P(x,y)\,g(x)\,dx, \quad y\in \R^3,\label{4.9bP}\\ 
&&{\cal P}g(y):= \int\limits_{\Omega} P(x,y)\,g(x)\,dx, \quad y\in \Omega,\label{4.9P}\\
&&{\cal R}g(y):=\int\limits_{\Omega} R(x,y)\,g(x)\,dx, \quad y\in \Omega. \label{4.9}
\end{eqnarray}
For $g\in {H}^s(\Omega)$, $s\in\R$, \eqref{4.9bP} is understood as 
${\mathbf P}\,g\,=\frac{1}{a}\;{\mathbf P}_\Delta\,g,$
where the Newtonian potential operator ${\mathbf P}_\Delta$ for the Laplace operator $\Delta$ is well defined in terms of the Fourier transform (i.e., as the pseudo-differential operator),
 on any space $H^s(\R^3)$. 
For  $g\in \tilde{H}^s(\Omega)$, and any $s\in\R$, definitions \eqref{4.9P} and \eqref{4.9} can be understood as
\begin{align}\label{PRdef}
{\cal P}g= \frac{1}{a}\;r_\Omega{\mathbf P}_\Delta\,g, \quad 
{\cal R}g= -\frac{1}{a}\;r_\Omega\nabla\cdot{\mathbf P}_\Delta\,(g\,\nabla a),
\end{align}
while for $g\in {H}^s(\Omega)$, $-\ha<s<\ha$, as \eqref{PRdef}
with $g$ replaced by $\widetilde E g$, where $\widetilde E:{H}^s(\Omega)\to \widetilde{H}^s(\Omega)$, $-\ha<s<\ha$, is the unique continuous extension operator related with the operator $\mathring E$ of extension by zero, cf. \cite[Theorem 2.16]{MikJMAA2011}. 

The single and the double layer surface potential operators, are defined as
\begin{eqnarray}
\label{3.6} && \dst
Vg(y):=-\int\limits_{\pO} P(x,y)\, g(x)\,dS_x,   \;\;\;\; y\not\in \pO,  \\
\label{3.7} && \dst Wg(y):=-\int\limits_{\pO}
\big[\,T(x,n(x),\pa_x)\,P(x,y)\,\big]\, \, g(x)\,dS_x,   \;\;\;\;
y\not\in \pO,\quad
\end{eqnarray}
where the integrals are understood in the distributional sense if
$g$ is not integrable.

The corresponding boundary integral (pseudodifferential) operators
of direct surface values of the single layer potential $\V$ and of
the double layer potential $\W$, and the co-normal derivatives of
the single layer potential $\Wp$ and of the double layer potential
$\L^+$ are
\begin{eqnarray}
\label{3.11} && \dst
{\cal V}\,g(y):=-\int\limits_{\pO} P(x,y)\, g(x)\,dS_x,  \\[1mm]
\label{3.12} && \dst {\cal W}\,g(y):=-\int\limits_{\pO}
\big[\,T^+_x\,P(x,y)\,\big]\,
\, g(x)\,dS_x,    \\[1mm]
\label{3.13} && \dst {\Wp}\,g(y):= -\int\limits_{\pO}
\big[\,T^+_y\,P(x,y)\,\big]\, g(x)\,dS_x,
        \\[1mm]
\label{3.14} && \dst {\cal L}^{+ }g(y):= T^+Wg(y),
\end{eqnarray}
where $y\in \pO$.

From definitions \eqref{3.2}, \eqref{3.6}, \eqref{3.7} one can obtain
representations of the parametrix-based potential operators in terms of their counterparts for $a=1$ (i.e.
associated with the  Laplace operator $\Delta$), which we equip with the subscript $\Delta$, cf. \cite{CMN-1},
\begin{eqnarray}
\label{3.d1} 
{\mathbf P}\,g\,=\frac{1}{a}\;{\mathbf P}_\Delta\,g,\quad  {\mathcal P}\,g\,=\frac{1}{a}\;{\mathcal P}_\Delta\,g,
  &&\quad
   {\mathcal R}\,g=-\,\frac{1}{a}\; \nabla\cdot{\mathcal P}_\Delta\,(g\,\nabla a),\qquad\\
\label{VWa}
Vg=\frac{1}{a}V_{_\Delta} g,&&\quad Wg=\frac{1}{a}W_{_\Delta}
(ag).
\end{eqnarray}
\begin{eqnarray}
\label{VWab1}&&\mathcal V g=\frac{1}{a}\mathcal V_{_\Delta} g,\quad \mathcal W
g=\frac{1}{a}\mathcal W_{_\Delta}
(ag),\\
\label{VWab3}&&\mathcal W\,^\prime g= {\mathcal W^{\,\prime}_{_\Delta}} g
+\left[a\frac{\partial }{\partial
n}\left(\frac{1}{a}\right)\right]\mathcal V_{_\Delta} g
,\\
\label{VWab4}
 &&{\mathcal L}^{\pm }g
 ={\mathcal L}_{_\Delta}(ag) +\left[a\frac{\partial }{\partial
n}\left(\frac{1}{a}\right)\right] W^\pm_{_\Delta} (ag).
\end{eqnarray}
Hence
 \begin{equation}\label{DV,DW=0}
\Delta(aVg)=0, \quad \Delta(aWg)=0\ \text{in}\ \Omega, \quad
\forall g\in H^s(\partial \Omega)\quad \forall s\in\R,
 \end{equation}
\begin{equation}\label{DeltaPg}
    \Delta(a\P g)=g\ \text{in}\ \Omega, \quad \forall
g\in \s{H}^s(\Omega)\quad \forall s\in\R,
\end{equation}

For $g_1\in H^{-\ha}(\pO)$, and  $g_2\in H^{\ha}(\pO)$, there hold
the following jump relations on $\pO$
\begin{eqnarray}
&& \dst [Vg_1(y)]^{+}= {\cal V}g_1(y)\label{3.8}
\\
&& \dst [Wg_2(y)]^{+}= - \ha\,g_2(y)  + {\cal W}g_2(y),\label{3.9}
\\
&& \dst [T(y,n(y),\pa_y)Vg_1(y)]^{+}= \ha\,g_1(y) +
{\Wp}g_1(y),\label{3.10}
\end{eqnarray}
where  $y\in \pO$.

The jump relations as well as mapping properties of potentials and
operators \eqref{3.6}-\eqref{4.9}  are well known for the case
$a=const$. Employing \eqref{3.d1}-\eqref{VWab4}, they were extended to the case of variable coefficient
$a(x)$ in \cite{CMN-1,CMN-2}, and in addition to
\eqref{3.8}-\eqref{3.10} some of them are presented in the Appendix
for convenience.

\section{The third Green identity and integral relations}\label{G3G}
We will apply in this section some limiting procedures (cf.
\cite{Miranda1970}, \cite[S. 3.8]{Hellwig1977}) to obtain the parametrix-based third Green identities. 

\begin{theorem}\label{3rdGA}
(i) If $u\in {H}^{1}(\Omega)$, then following {\em third Green identity} holds,
\begin{equation}
u+{\cal R}u +W\gamma^+u=\P\check{A}u
 \quad \mbox{in}\ \Omega,\label{4.G3til}
\end{equation}
where the operator $\check A$ is defined in \eqref{checkA}, and 
for  $u\in C^1(\overline \Omega)$,
\begin{eqnarray}\label{LPtil}
 \P\check{A}u(y):=   \langle \check{A}u,P(\cdot,y) \rangle_\Omega=-\E(u,P(\cdot,y))
 =-\int_{\Omega}a(x) \nabla u(x) \cdot\nabla_x P(x,y)\,dx.\qquad
\end{eqnarray}

(ii) If $Au=\tilde f|_\Omega$ in $\Omega$, where $\tilde f\in \s{H}^{-1}(\Omega)$,
then the {\em generalised third Green identity} takes form,
\begin{eqnarray}
u+{\cal R}u - VT^+(\tilde f,u)  +W\gamma^+u=
 {\cal P}\tilde f  &&\mbox{in }\Omega. \qquad\label{4.2Gen}
 \end{eqnarray}
\end{theorem}
\begin{proof}
(i) Let first $u\in\mathcal D(\overline \Omega)$. Let $y\in\Omega$,  $B_\epsilon(y)\subset \Omega$ be a ball centered in $y$ with sufficiently small radius $\epsilon$, and $\Omega_\epsilon:=\Omega\setminus\overline{B_\epsilon}(y)$. For the fixed $y$, evidently, $P(\cdot,y)\in \mathcal D(\overline{\Omega_\epsilon})\subset H^{1,0}(A;\Omega_\epsilon)$ and has the coinciding classical and canonical conormal derivatives on $\pO_\epsilon$. Then from \eqref{3.1} and the first Green identity \eqref{Greentilde} employed for $\Omega_\epsilon$ with $v=P(\cdot,y)$ we obtain
\begin{multline}\label{4.G3tile}
 - \int_{\partial B_\epsilon(y)} T^{+}_x P(x,y)\gamma^+u(x) dS(x) - \int_\pO T^{+}_x P(x,y)\gamma^+u(x)  dS(x)
 +\int_{\Omega_\epsilon}\, u(x)R(x,y)dx \\
 =- \int_{\Omega_\epsilon}a(x) \nabla u(x) \cdot\nabla_x P(x,y)\,dx.
 \end{multline}
 Taking limits as $\epsilon\to 0$, equation \eqref{4.G3tile} reduces to the third Green identity \eqref{4.G3til}-\eqref{LPtil} for any $u\in \mathcal D(\overline \Omega)$.
Taking into account the density of $\mathcal D(\overline \Omega)$ in $H^{1}(\Omega)$, and the mapping properties of the integral potentials, see Appendix, we obtain that \eqref{4.G3til} holds true also for any
$u\in H^{1}(\Omega)$.

(ii) Let  $\{\tilde f_k\}\in\mathcal D(\Omega)$ be a sequence converging to 
$\tilde f$ in $\widetilde H^{-1}(\Omega)$ as $k\to\infty$. 
Then, according to Theorem~\ref{Tseq}, 
there exists a sequence  $\{u_k\}\in\mathcal D(\overline\Omega)$ converging to $u$ in $H^{1}(\Omega)$ such that 
$Au_k=r_\Omega\tilde f_k$ and $T^+(u_k)=T^+(\tilde f_k,u_k)$ converges to $T^+(\tilde{f},u)$ in ${H^{-\ha}(\partial\Omega)}$.
For such $u_k$ we have by \eqref{LPtil} and \eqref{Tgen},
 \begin{multline*} 
  \P\check{A}u_k(y)=\frac{1}{a(y)}\nabla_y\cdot\int_{\Omega} a(x) P_\Delta(x,y) \nabla u_k(x)  \,dx\\
=-\lim_{\epsilon\to 0}\int_{\Omega_\epsilon} a(x) \nabla u_k(x) \cdot\nabla_x P(x,y)\,dx 
=-\lim_{\epsilon\to 0} \mathcal E_{\Omega_\epsilon}(u_k,P(\cdot, y))\\
=\lim_{\epsilon\to 0}\left[\int_{\Omega_\epsilon} \tilde f_k P(x,y)\,dx 
 - \int_{\partial B_\epsilon(y)} P(x,y)T^+u_k(x) dS(x) - \int_\pO P(x,y)T^+u_k(x)  dS(x)\right]  
={\cal P}\tilde f_k + VT^+u_k(y).\quad
 \end{multline*} 
Taking limits as $k\to\infty$, we obtain 
$
\P\check{A}u(y)={\cal P}\tilde f + VT^+(\tilde f,u),
$
which substitution to \eqref{4.G3til} gives \eqref{4.2Gen}.
\end{proof}

For some functions $\tilde f$, $\Psi$, $\Phi$, let us consider a more
general "indirect" integral relation, associated with
\eqref{4.2Gen},
\begin{eqnarray}
 u+{\cal R}u - V\Psi +W\Phi &=&
 {\cal P}\tilde f\ \mbox{ in }\Omega. \label{4.2nd}
\end{eqnarray}
The following statement extends Lemma 4.1 from \cite{CMN-1}, where
the corresponding assertion was proved for $\tilde f\in L_2(\Omega)$.

\begin{lemma}\label{IDequivalenceGen}
Let $u\in H^1(\Omega)$, $\Psi\in H^{-\ha}(\pO)$, $\Phi\in
H^{\ha}(\pO)$, and $\tilde f\in \s{H}^{-1}(\Omega)$ satisfy
(\ref{4.2nd}). Then
\begin{align}
\label{2.6'gen}Au&= r_\Omega \tilde f \text{in}  \Omega,
\\
\label{difference}
 r_\Omega V(\Psi -T^+(\tilde f,u) ) - r_\Omega W(\Phi - \gamma^+u) &= 0 \text{in}  \Omega,
\\
\gamma^+u +\gamma^+{\cal R}u - \V \Psi -\ha\Phi +\W \Phi&=
 \gamma^+{\cal P}\tilde f \text{on}\partial\Omega, \qquad\label{4.2u+gen}\\
 T^+(\tilde f,u)+T^+{\cal R}u -\ha\Psi - {\Wp}\Psi +\L^+ \Phi &=
  T^+(\tilde f+\mathring E\,{\mathcal R}_*\tilde f,{\cal P}\tilde f)  \text{on}  \partial\Omega,\quad \label{4.2T+gen}
\end{align}
where ${\mathcal R}_*{\tilde f}\in L_2(\Omega)$ is defined as
 \begin{equation}\label{R^0}
 {\mathcal R}_*{\tilde f}:=-\sum_{j=1}^3\partial_j[(\partial_j a)\mathcal P{\tilde f}].
 \end{equation}
\end{lemma}
\begin{proof} 
Subtracting  \eqref{4.2nd} from identity \eqref{4.G3til}, we obtain
\begin{equation}
\label{4.12} V\Psi(y)-W(\Phi-\gamma^+u)(y) ={\cal P}[\check{A}u-\tilde f](y),
\;y\in \Omega .
\end{equation}
Multiplying equality \eqref{4.12} by $a(y)$, applying the Laplace
operator $\Delta$ and taking into account \eqref{DV,DW=0},
\eqref{DeltaPg}, we get
$
r_\Omega\tilde f=r_\Omega\check{A}u=Au$ in $\Omega.
 $
This means $\tilde f$ is an extension of the distribution $Au\in H^{-1}(\Omega )$ to
$\s{H}^{-1}(\Omega )$, and $u$ satisfies  \eqref{2.6'gen}. Then \eqref{Tgen}
implies
 \begin{eqnarray}
 \label{Tgen1}
{\cal P}[\check{A}u-\tilde f](y)&=&\langle \check{A}u-\tilde f,
P(\cdot,y)\rangle_\Omega
=-\langle
 T^+(\tilde f,u)\,,\,P(\cdot,y)
\rangle _{\pO} =VT^+(\tilde f,u), \quad y\in\Omega .\qquad
\end{eqnarray}
Substituting \eqref{Tgen1} into \eqref{4.12} leads to
\eqref{difference}.

Equation \eqref{4.2u+gen} is implied by \eqref{4.2nd}, \eqref{3.8} and \eqref{3.9}.

To prove \eqref{4.2T+gen}, let us first remark that
\begin{equation}\label{LPf}
 A\P\tilde f=\tilde f+{\cal R}_*\tilde f\  \text{in}\ \Omega,
\end{equation}
which implies, due to \eqref{2.6'gen},
$
    A(\P\tilde f-u)={\cal R}_*\tilde f\ \text{in}\ \Omega,
$
where ${\cal R}_*$ is defined by \eqref{R^0} and thus ${\cal R}_*\tilde f\in L_2(\Omega)$. 
Then $A(\P\tilde f-u)$ can be canonically extended (by zero) to 
$\tilde A(\P\tilde f-u)=\mathring E\, {\mathcal R}_*\tilde f\in\s{H}^0(\Omega)\subset\s{H}^{-1}(\Omega)$. 
This implies that there exists a
canonical co-normal derivative of $(\P\tilde f-u)$, for which, due to
\eqref{Tcandef} and \eqref{Ltil}, we have
 \begin{multline*}
 \langle T^+(\P\tilde f-u),w\rangle_{\partial\Omega}=
 \langle \tilde A(\P\tilde f-u)-\check{A}\P\tilde f +\check{A}u,\gamma^{-1}w\rangle_\Omega=
 \langle \mathring E\,{\mathcal R}_*\tilde f-\check{A}\P\tilde f +\check{A}u,\gamma^{-1}w\rangle_\Omega\\
 =\langle \mathring E\,{\mathcal R}_*\tilde f+\tilde f-\tilde f-\check{A}\P\tilde f +\check{A}u,\gamma^{-1}w\rangle_\Omega=
 \langle \tilde f+\mathring E\,{\mathcal R}_*\tilde f-\check{A}\P\tilde f +\check{A}u-\tilde f,\gamma^{-1}w\rangle_\Omega\\
 =\langle T^+(\tilde f+\mathring E\,{\mathcal R}_*\tilde f,\P\tilde f)
 -T^+(\tilde f,u),w\rangle_{\partial\Omega}\quad
 \forall w\in H^\ha(\partial\Omega),
 \end{multline*}
where  $\tilde f+\mathring E\,{\mathcal R}_*\tilde f\in\s{H}^{-1}(\Omega)$ is an
extension of $A\P\tilde f$ associated with \eqref{LPf}. That is,
\be\label{TPf-u}
T^+(\P\tilde f-u)=T^+(\tilde f+\mathring E\,{\mathcal R}_*\tilde f,\P\tilde f) - T^+(\tilde f,u)\ \text{on} \partial\Omega.
\ee
From \eqref{4.2nd} we have ${\cal P}\tilde f-u={\cal R}u - V\Psi +W\Phi$
in $\Omega$. Substituting this in the left hand side of
\eqref{TPf-u} and taking into account jump relation \eqref{3.10}, we
arrive at \eqref{4.2T+gen}
\end{proof} 

\begin{remark}
If $\tilde f\in \widetilde H^{-1/2}(\Omega)\subset \widetilde H^{-1}(\Omega)$, then 
$\tilde f+\mathring E\,{\mathcal R}_*\tilde f\in \widetilde H^{-1/2}(\Omega)$ 
as well, which implies $\tilde f+\mathring E\,{\mathcal R}_*\tilde f=\tilde A{\cal P}\tilde f$ and 
\begin{equation}\label{TfPTP}
T^+(\tilde f+\mathring E\,{\mathcal R}_*\tilde f,{\cal P}\tilde f)=T^+(\tilde A{\cal P}\tilde f,{\cal P}\tilde f)=T^+{\cal P}\tilde f.
\end{equation}
Furthermore, if the hypotheses of Lemma~\ref{IDequivalenceGen} are satisfied, then \eqref{2.6'gen} implies $u\in H^{1,-1/2}(\Omega,A)$ and $ T^+(\tilde f,u)= T^+(\tilde Au,u)= T^+u$.  
Henceforth, \eqref{4.2T+gen} takes the familiar form, cf. \cite[equation (4.5)]{CMN-1},
$$
 T^+u+T^+{\cal R}u -\ha\Psi - {\Wp}\Psi +\L^+ \Phi =
  T^+{\cal P}\tilde f \text{ on } \partial\Omega.
$$
\end{remark}

\begin{remark}
Let $\tilde f\in  \widetilde H^{-1}(\Omega)$ and a sequence $\{\phi_i\}\in \widetilde H^{-1/2}(\Omega)$ converge to $\tilde f$ in  $\widetilde H^{-1}(\Omega)$. By the continuity of operators \eqref{T3.1P1} and \eqref{T3.1P3} in the Appendix, estimate \eqref{estimate} and relation \eqref{TfPTP} for $\phi_i$, we obtain that 
$$
T^+(\tilde f+\mathring E\,{\mathcal R}_*\tilde f,{\cal P}\tilde f)
=\lim_{i\to\infty}T^+(\phi_i+\mathring E\,{\mathcal R}_*\phi_i,{\cal P}\phi_i)
=\lim_{i\to\infty}T^+{\cal P}\phi_i
$$
in $H^{-1/2}(\partial\Omega)$, cf. also Theorem~\ref{Tseq}.
\end {remark}

Lemma \ref{IDequivalenceGen} and the third Green identity \eqref{4.2Gen}
imply the following assertion.
\begin{corollary}\label{IDequivGenDir}
If $u\in H^{1}(\Omega)$ and $\tilde f\in \s{H}^{-1}(\Omega)$ are such that $Au=r_\Omega\tilde f$ in $\Omega$, then
 \begin{align}
 \frac{1}{2}\gamma^+u +\gamma^+{\cal R}u - \V T^+(\tilde f,u)  +\W
 \gamma^+u&=
 \gamma^+{\cal P}\tilde f \mbox{ on }\pO, \qquad\label{4.2u+Gen}\\
 \frac{1}{2}T^+(\tilde f,u) +T^{+}{\cal R}u -
 {\Wp}T^+(\tilde f,u)  +\L^+ \gamma^+u 
 &=T^+(\tilde f+\mathring E\,{\mathcal R}_*\tilde f,{\cal P}\tilde f) \mbox{ on } \pO.\qquad\quad
\label{4.2T+Gen}
\end{align}
\end{corollary}

The following statement is well known, see e.g. Lemma
4.2 in \cite{CMN-1} and references therein.
\begin{lemma}\label{VW0s}\quad
\begin{itemize}
\item[(i)]  If $\Psi^*\in {H}^{-\ha}(\pO)$ and
 $r_\Omega V\Psi^*=0$ in $\Omega,$
then $\Psi^*=0$.
 \item[(ii)]  If $\Phi^*\in {H}^{\ha}(\pO)$ and
 $r_\Omega  W\Phi^*=0 $ in $\Omega,$
then $\Phi^*=0$.
\end{itemize}
\end{lemma}

\begin{theorem}\label{GIequivalenceH-1}
Let $\tilde f\in \s{H}^{-1}(\Omega)$. A function $u\in H^1(\Omega)$ is a
solution of PDE $Au=\tilde f|_\Omega$ in $\Omega$ if and only if it is a
solution of BDIDE (\ref{4.2Gen}).
\end{theorem}
\begin{proof} If $u\in H^1(\Omega)$ solves PDE $Au=\tilde f|_\Omega$ in
$\Omega$, then it satisfies (\ref{4.2Gen}).  On the other hand, if
$u $ solves BDIDE \eqref{4.2Gen}, then using Lemma
\ref{IDequivalenceGen} for $\Psi=T^+(\tilde f,u) $, $\Phi=\gamma^+u$
completes the proof.
\end{proof}
 
\section{Segregated BDIE systems for the Dirichlet problem}

Let us consider the
{\bf Drichlet Problem:} {\em Find a function $u\in H^1(\Omega)$ satisfying equations}
\begin{eqnarray}
\label{2.6} && A\,u=f  \;\;\; \mbox{\rm in}\;\;\;\; \Omega,
\\ 
 \label{2.7} &&  \gamma^+u=\varphi_0  \;\;\; \mbox{\rm on}\;\;\;\; \pO,
\end{eqnarray}
{\em where} $ \varphi_0 \in H^{\ha}(\pO)$,  $ f\in {H}^{-1}(\Omega)$.

Equation \eqref{2.6} is understood in the distributional sense
\eqref{Ldist} and the Dirichlet boundary condition \eqref{2.7} in the trace sense.
The following assertion is well-known and can be proved e.g. using variational settings and the Lax-Milgram lemma.
\begin{theorem}\label{Rem1}
The Dirichlet  problem \eqref{2.6}-\eqref{2.7} is uniquely solvable in ${H}^{1}(\Omega)$. 
The solution is $u=(\mathcal A^D)^{-1}(f,\varphi_0)^\top$,
where the inverse operator, $(\mathcal A^D)^{-1}: H^{\ha}(\pO) \times H^{-1}(\Omega)\to H^{1}(\Omega)$, to the left hand side operator, $\mathcal A^D: H^{1}(\Omega)\to H^{\ha}(\pO) \times H^{-1}(\Omega)$, of the Dirichlet problem \eqref{2.6}-\eqref{2.7}, is continuous.
\end{theorem}

\subsection{BDIE formulations and equivalence to the Dirichlet problem}

Let us consider reduction the Dirichlet problem \eqref{2.6}-\eqref{2.7} with $f\in {H}^{-1}(\Omega)$,  for $u\in H^1(\Omega)$, to two different {\em segregated}  Boundary-Domain Integral Equation (BDIE)
systems.
Corresponding formulations for the
mixed problem for $u\in H^{1,0}(\Omega;\Delta)$ with  $  f\in
L_2(\Omega)$ were introduced and analysed in \cite{CMN-1, CMN-2, MikMMAS2006}.

Let $\tilde f\in \s{H}^{-1}(\Omega)$ be an extension of $f\in {H}^{-1}(\Omega)$ (i.e., $f=r_\Omega\tilde f$), which always exists, see \cite[Lemma 2.15 and Theorem 2.16]{MikJMAA2011}. Let us represent in  \eqref{4.2Gen}, \eqref{4.2u+Gen} and \eqref{4.2T+Gen} the generalised co-normal derivative and the trace of the function $u$ as
$$
T^+(\tilde f,u)=\psi,\qquad \gamma^+u=\varphi_0,
$$
and will regard the new unknown function $\psi\in{H}^{-\frac{1}{2}}(\partial\Omega)$ as formally {\em segregated} of $u$.
Thus we will look for the couple
$
(u, \psi)
\in H^1(\Omega)\times {H}^{-\frac{1}{2}}(\partial\Omega).
$

\paragraph{BDIE system (D1)}
To reduce the Dirichlet BVP \eqref{2.6}-\eqref{2.7} to the BDIE system (D1), we will use equation \eqref{4.2Gen} in $\Omega$ and
equation \eqref{4.2u+Gen} on $\pO$.
Then we
arrive at the following system, (D1), of the boundary-domain integral equations,
\begin{align}
\label{4.5GT1}
 u+{\cal R}u - V\psi =&\F^{D1}_1 \text{ in}\ \Omega,     \\
 \gamma^+{\cal R} u -  {\cal V}\psi =& \F^{D1}_2 
\text{ on}\ \pO, \label{4.5GT2}
\end{align}
where
\begin{equation}\label{FD1}
\F^{D1}=\left[\begin{array}{l}
   \F^{D1}_1\\[1ex]
   \F^{D1}_2
   \end{array}\right]
   =\left[\begin{array}{l}
  F_0^D\\[1ex]
  \gamma^+F_0^D-\varphi_0 
  \end{array}\right]
\text{ and }
 F_0^D:={\cal P}\tilde f -W\varphi_0
\;\;\;\text{in}\ \Omega.
\end{equation}
Note that for $\varphi_0\in H^{\ha}(\pO)$, we have the inclusion
$F_0^D\in H^{1}(\Omega)$ if $\tilde f\in \s{H}^{-1}(\Omega)$ due to the mapping properties of
the Newtonian (volume) and layer potentials, cf.  
(\ref{T3.1P1}), (\ref{WHs1}).
%

\paragraph{BDIE system (D2)}
To obtain a segregated BDIE system of {\em the second kind}, (D2), we
will use equation \eqref{4.2Gen} in $\Omega$ and  equation
\eqref{4.2T+Gen} on $\pO$
Then we arrive at the following BDIE system (D2),
\begin{eqnarray}
\label{2.44} u+{\cal R}u - V\psi  =\F^{D2}_1 &\text{in}&
  \Omega,   \\
\label{2.45} \frac{1}{2}\,\psi+T^+{\cal R}u  - {\Wp}\psi =
\F^{D2}_2&\text{on}& \pO,
\end{eqnarray}
where
\begin{equation}\label{FD2}
\F^{D2}=\left[\begin{array}{l}
   \F^{D2}_1\\[1ex]
   \F^{D2}_2
   \end{array}\right]
   =\left[\begin{array}{l}
  {\cal P}\tilde f-W\varphi_0\\[1ex]
  T^+(\tilde f+\mathring E\,{\mathcal R}_*\tilde f,{\cal P}\tilde f)-\L^+ \varphi_0
  \end{array}\right].
\end{equation}
Due to the mapping properties of the operators involved in
\eqref{FD2} we have $\F^{D2}\in H^{1}(\Omega )\times H^{-\ha}(\pO)$.

Let us prove that BVP \eqref{2.6}--\eqref{2.7} in $\Omega$ is
equivalent to both systems of BDIEs, (D1) and (D2).

\begin{theorem}
\label{Deqin}
Let 
$ \varphi_0 \in H^{\frac{1}{2}}(\partial\Omega)$,
$f\in  H^{-1}(\Omega)$, and   $\tilde f\in  \widetilde H^{-1}(\Omega)$ is such that $r_{_\Omega}\tilde f=f$.
%
 \begin{enumerate}
 \item[\rm (i)]
 If a function $u\in {H}^{1}(\Omega)$ solves the Dirichlet BVP \eqref{2.6}--\eqref{2.7}, then the couple $(u, \psi)\in H^1 (\Omega)\times { H}^{-\ha}(\pO)$,
    where 
    \begin{equation}
    \label{upsiphiGT}
    \psi=T^+ (\tilde f,u) \;\;\;\;\text{on}\;\;\;\; \pO, \qquad
    \end{equation}
      solves the BDIE systems (D1) and (D2).
 \item[\rm (ii)]
 If a couple $(u, \psi)\in {H}^{1}(\Omega)\times {H}^{-\frac{1}{2}}(\partial\Omega)$ solves one of the BDIE systems, (D1) or (D2), then this solution is unique and solves the other system, while $u$ solves the Dirichlet BVP, and $\psi$ satisfies (\ref{upsiphiGT}).
 \end{enumerate}
 \end{theorem}
 \begin{proof} 
(i) Let $u\in H^1(\Omega)$ be a solution to BVP
\eqref{2.6}--\eqref{2.7}. It is unique due to Theorem \ref{Rem1}.
Setting $\psi$  by \eqref{upsiphiGT} evidently implies $\psi\in
{H}^{-\ha}(\pO)$. Then it immediately follows from Theorem~\ref{GIequivalenceH-1} and relations
 \eqref{4.2u+Gen} and \eqref{4.2T+Gen} that the couple $(u,\psi)$ solves
systems (D1) and (D2) with the right
hand sides \eqref{FD1} and \eqref{FD2}, respectively,  which completes the proof of
item (i).

(ii) Let now a couple $(u,\psi)\in H^1(\Omega)\times { H}^{-\ha}(\pO)$
solve BDIE system \eqref{4.5GT1}-\eqref{4.5GT2}. Taking trace of
equation \eqref{4.5GT1} on $\pO$  and subtracting equation
\eqref{4.5GT2} from it, we obtain,
\begin{equation}\label{DGT}
     \gamma^+u(y)=\varphi_0(y),\quad y\in \pO,
\end{equation}
i.e. $u$ satisfies the Dirichlet condition \eqref{2.7}.

Equation \eqref{4.5GT1} and Lemma \ref{IDequivalenceGen} with
$\Psi=\psi$, $\Phi=\varphi_0$ imply that $u$ is a solution of PDE
\eqref{2.6} and
 $$
V\Psi^* (y)- W\Phi^* (y)=0,\quad y\in \Omega,
$$
where $\Psi^*=\psi -T^+ (\tilde f,u) $ and $\Phi^*=\varphi_0 - \gamma^+u$. Due to equation \eqref{DGT}, $\Phi^*=0$. Then Lemma \ref{VW0s}(i)
implies $\Psi^*=0$, which completes the proof of condition
\eqref{upsiphiGT}.
Thus $u$ obtained from solution of BDIE system (D1) solves the Dirichlet problem and hence, by item (i) of the theorem, $(u,\psi)$ solve also BDIE system (D2).  

Due to \eqref{FD1}, the BDIE system
\eqref{4.5GT1}-\eqref{4.5GT2} with zero right hand side can be considered as obtained for $\tilde f=0$, $\varphi_0=0$, implying that its solution is given by a solution of the homogeneous BVP \eqref{2.6}--\eqref{2.7}, which is zero by Theorem \ref{Rem1}. This implies uniqueness of solution of the inhomogeneous BDIE system
\eqref{4.5GT1}-\eqref{4.5GT2}.

 Let now a couple $(u,\psi)\in H^{1}(\Omega)\times H ^{-\ha}(\pO)$
solves BDIE system \eqref{2.44}-\eqref{2.45}.
Lemma \ref{IDequivalenceGen} for equation \eqref{2.44} implies that
$u$ is a solution of equation \eqref{2.1}, and equations
\eqref{difference} and \eqref{4.2T+gen} hold for $\Psi=\psi$ and
$\Phi=\varphi_0$. Subtracting \eqref{4.2T+gen} from equation
\eqref{2.45} gives
\begin{equation}
 \Psi^* :=\psi -T^+ (\tilde f,u) =0\;\;\;\;\text{on}\;\;\;\;
 \pO,\label{PsiGen}
\end{equation}
that is, equation (\ref{upsiphiGT}) is proved.

Equations \eqref{difference} and \eqref{PsiGen} give
 $ W\Phi^* =0$ in $\Omega$,
 where
$\Phi^*=\varphi_0 - \gamma^+u$. Then Lemma \ref{VW0s}(ii) implies
$\Phi^*=0$ on $\pO$. This means that $u$ satisfies the Dirichlet
condition \eqref{2.7}.
Thus $u$ obtained from solution of BDIE system (D2) solves the Dirichlet problem and hence, by item (i) of the theorem, $(u,\psi)$ solve also BDIE system (D1).  

Due to \eqref{FD2}, the BDIE system
\eqref{2.44}-\eqref{2.45} with zero right hand side can be considered as obtained for $\tilde f=0$, $\varphi_0=0$, implying that its solution is given by a solution of the homogeneous BVP \eqref{2.6}--\eqref{2.7}, which is zero by Theorem \ref{Rem1}. This implies uniqueness of solution of the inhomogeneous BDIE system
\eqref{2.44}-\eqref{2.45}.
\end{proof} 

\begin{remark}
For a given function $f\in  H^{-1}(\Omega)$, its extension  $\tilde f\in  \widetilde H^{-1}(\Omega)$ is not unique. Nevertheless, since solution of the Dirichlet BVP \eqref{2.6}--\eqref{2.7} does not depend on this extension, equivalence Theorem~\ref{Deqin}(ii) implies that $u$ in the solution of BDIE systems (D1) and (D2) does not depend on the particular  choice of extension $\tilde f$. However, $\psi$ does obviously depends on the choice of $\tilde f$, see \eqref{upsiphiGT}.
\end{remark}

\subsection{BDIE system operators invertibility, for the Dirichlet problem}

BDIE systems (D1) and (D2) can be written as
\begin{equation*}
  \mathfrak{D}^1\U=\F^{D1} \text{ and }\  \mathfrak{D}^2\U=\F^{D2},
\end{equation*}
respectively. Here $\U^D:=(u, \psi)^\top\in H^1(\Omega)\times { H}^{-\ha}(\pO)$,
\begin{equation}\label{D1-D2}
\mathfrak{D}^1:= \left[
\begin{array}{cc}
I-{\cal R} & -V  \\
\gamma^+{\cal R}  & -{\cal V}
\end{array}
\right],\quad 
   \mathfrak{D}^2:= \left[
 \begin{array}{ccc}
 I+{\cal R} & -V  \\[1ex]
   T^+{\cal R} & \displaystyle \,\frac{1}{2}\,I- {\Wp}
 \end{array}
 \right],
\end{equation}
while $\F^{D1}$ and $\F^{D2}$ are given by \eqref{FD1} and \eqref{FD2}.

Due to the mapping properties of the operators participating in definitions of the operators $\mathfrak{D}^1$ and $\mathfrak{D}^2$ as well as the right hand sides $\F^{D1}$ and $\F^{D2}$ (see \cite{CMN-1, MikMMAS2006} and
the Appendix), we have 
$\F^{D1}\in H^1(\Omega)\times H^{\ha}(\pO)$, $\F^{D2}\in H^1(\Omega)\times H^{-\ha}(\pO)$, 
while the operators
  \begin{eqnarray}
\mathfrak{D}^1 
&:& H^1 (\Omega)\times { H}^{-\ha}(\pO)\to H^1 (\Omega)\times H^{\ha}(\pO), \label{AGcGcont}\\
\mathfrak{D}^2
&:& H^1 (\Omega)\times { H}^{-\ha}(\pO)\to H^1 (\Omega)\times H^{-\ha}(\pO)\label{AGcTcont}
  \end{eqnarray}
are continuous. 
Due to Theorem \ref{Deqin}(ii), operator \eqref{AGcGcont} and \eqref{AGcTcont} are injective.

\begin{theorem}\label{AGTinvs}
Operators \eqref{AGcGcont} and \eqref{AGcTcont} are continuous and
continuously invertible.
\end{theorem}
\begin{proof} 
The continuity is proved above. 
To prove the invertibility of operator \eqref{AGcGcont}, let us consider the  operator
\begin{equation*}
\mathfrak{D}^1_0:= \left[
\begin{array}{cc}
I &        -V            \\
 0 & \,-{\cal V}
\end{array}
\right]\ .
\end{equation*}
As a result of compactness properties of the operators ${\cal R}$
and $\gamma^+{\cal R}$ (see Corollary \ref{B.3} in the Appendix), the operator
$\mathfrak{D}^1_0$ is a compact perturbation of operator \eqref{AGcGcont}.
The operator $\mathfrak{D}^1_0$ is an upper triangular matrix operator
with the following scalar diagonal invertible operators
\begin{eqnarray*}
I&\;:\;& H^1 (\Omega)\to H^1 (\Omega),\\
{\cal V} & \;:\;&   H^{-\ha}(\pO)\to H^{\ha}(\pO),
\end{eqnarray*}
cf. \cite[ Ch. XI, Part B, \S 2, Theorem 3]{DaLi4} for $\V$. This
implies that
$$
\mathfrak{D}^1_0 \;:\; H^1 (\Omega)\times { H}^{-\ha}(\pO) \to H^1
(\Omega)\times { H}^{\ha}(\pO)
$$
is an invertible  operator. Thus   \eqref{AGcGcont} is a Fredholm operator with zero index. 
The injectivity of
operator \eqref{AGcGcont}, which is already proved, completes the theorem
proof for operator \eqref{AGcGcont}.

 The operator
$$
  \mathfrak{D}^2_0:= \left[
\begin{array}{ccc}
I  & -V  \\
0  & \displaystyle \frac{1}{2}\,I
\end{array}
\right].
$$
is a compact perturbation of  operator
\eqref{AGcTcont} due to compactness properties of the operators
$\cal R$ and $\W$, see \cite{CMN-1, CMN-2, MikMMAS2006} and Corollary
\ref{B.3} from the Appendix. The invertibility of operator \eqref{AGcTcont} then follows by the arguments
similar to those for operator \eqref{AGcGcont}.
\end{proof}

\section{Segregated BDIE systems for the Neumann Problem}
Let us consider the 
{\bf Neumann Problem:} {\em Find a function $u\in H^1(\Omega)$ satisfying equations}
\begin{align}
\label{2.6N} & A\,u=r_\Omega\tilde f  \;\;\; \mbox{\rm in}\;\;\;\; \Omega,
\\ 
 \label{2.8} &  T^+(\tilde f,u)=\psi_0  \;\;\; \mbox{\rm on}\;\;\;\; \pO,
\end{align}
{\em where} $ \psi_0 \in H^{-\ha}(\pO)$,  $\tilde f\in \widetilde{H}^{-1}(\Omega)$. 

Equation \eqref{2.6N} is understood in the distributional sense
\eqref{Ldist} and Neumann boundary condition \eqref{2.8} in the weak sense \eqref{Tgen}.
The following assertion is well-known and can be proved e.g. using variational settings and the Lax-Milgram lemma.
\begin{theorem}\label{Rem1N} \

 (i) The Neumann homogeneous problem, associated with \eqref{2.6N}-\eqref{2.8}, admits only one linearly independent solution $u^0=1$ in ${H}^{1}(\Omega)$.
 
 (ii) The non-homogeneous Neumann problem \eqref{2.6N}-\eqref{2.8} is solvable if only if the following solvability condition is satisfied
\be\label{3.suf}
\langle \tilde f,u^0\rangle_\Omega -\langle\psi_0,\gamma^+u^0\rangle_\pO=0.
\ee
\end{theorem}
\subsection{BDIE formulations and equivalence to the Neumann problem}

We will explore different possibilities of
reducing the Neumann problem \eqref{2.6N}-\eqref{2.8} to a BDIE system. 
Let us represent in  \eqref{4.2Gen}, \eqref{4.2u+Gen} and \eqref{4.2T+Gen} the generalised co-normal derivative and the trace of the function $u$ as
$$
T^+(\tilde f,u)=\psi_0,\qquad \gamma^+u=\varphi,
$$
and will regard the new unknown function $\varphi\in{H}^{\frac{1}{2}}(\partial\Omega)$ as formally {\em segregated} of $u$.
Thus we will look for the couple
$
(u, \varphi)
\in H^1(\Omega)\times {H}^{\frac{1}{2}}(\partial\Omega).
$
\paragraph{BDIE system (N1)}
First, using equation \eqref{4.2Gen} in $\Omega$ and
equation \eqref{4.2T+Gen} on $\partial\Omega$,  we arrive
at the following BDIE system (N1) of two equations for the couple of unknowns,
$(u, \varphi)$,
\begin{eqnarray}
\label{Basic1}
 u +{\mathcal R}u  +W\varphi &=& \F^{N1}_1\quad \mbox{in}\quad   \Omega,
\\[1ex]
 \label{Basic2}
  T^+{\mathcal R} u +\mathcal L^+\varphi
  &=&
  \F^{N1}_2 \quad  \mbox{on}\quad  \partial\Omega,
\end{eqnarray}
where
\begin{equation}\label{FN1}
\F^{N1}=\left[\begin{array}{l}
  \F^{N1}_1\\[1ex]
  \F^{N1}_2
  \end{array}\right]
=\left[\begin{array}{l}
  {\cal P}\tilde f + V\psi_0\\[1ex]
  T^+(\tilde f+\mathring E\,{\mathcal R}_*\tilde f,{\cal P}\tilde f)-\dfrac{1}{2}\psi_0+{\Wp}\psi_0
  \end{array}\right].
\end{equation}
Due to the mapping properties of the operators involved in
\eqref{FN1} we have $\F^{N1}\in H^{1}(\Omega )\times H^{-\ha}(\pO)$.

\paragraph{BDIE system (N2)}
If we use equation \eqref{4.2Gen} in $\Omega$ and equation \eqref{4.2u+Gen} on $\partial\Omega$, we arrive for the couple $(u,  \varphi)$ at the following  BDIE system (N2) of two equations of the second kind,
\begin{eqnarray}
\label{Basic1-2}
 u +{\mathcal R}u     +W  \varphi
 &=& \mathcal F^{N2}_1  \qquad \mbox{in}\quad   \Omega,
\\[1ex]
 \label{Basic2-2}
 \frac{1}{2}\varphi+\gamma^+{\mathcal R} u   +{\mathcal W} \varphi
 &=&    \mathcal F^{N2}_2,  \quad   \mbox{on}\quad  \partial\Omega.
\end{eqnarray}
where 
\begin{align}
 \label{FN2}
\mathcal F^{N2}=\left[\begin{array}{l}
   \F^{N2}_1\\[1ex]
   \F^{N2}_2
   \end{array}\right]
   =\left[\begin{array}{l}
      F_0^N\\[1ex]
      \gamma^+F_0^N
      \end{array}\right],\quad
F_0^N:={\mathcal P} \tilde f + V\psi_0  \quad   \mbox{in}\quad   \Omega.
 \end{align}
Due to the mapping properties of the operators involved in
\eqref{FN2}, we have $\F^{N2}\in H^{1}(\Omega )\times H^{\ha}(\pO)$.

\begin{theorem}\label{equivalenceN}
Let   $\psi_0 \in
H^{-\frac{1}{2}}({\partial\Omega}) $ and $\tilde f\in \widetilde H^{-1}(\Omega )$.

(i) If a function $u\in H^1(\Omega )$ solves the Neumann  problem
 \eqref{2.6N}-\eqref{2.8} then the couple $(u,\varphi)$ with
$\varphi=\gamma^+u\in H^{\frac{1}{2}}({\partial\Omega})$ solves 
BDIE systems {\rm(N2)} and  {\rm(N1)}.

(ii) Vice versa, if a
couple $(u,\varphi)\in H^1(\Omega )\times
H^{\frac{1}{2}}({\partial\Omega})$ solves one of the
BDIE systems, {\rm(N1)}  or {\rm(N2)}, then the couple solves the other one BDE system and $u$ solves the Neumann
problem  \eqref{2.6N}-\eqref{2.8} and $\gamma^+u=\varphi$.

(iii) The homogeneous BDIE systems {\rm(N1)} and
 {\rm(N2)} have unique linear independent solution\, $\mathcal U^0=(u^0,\varphi^0)^\top=(1,1)^\top$ in $H^1(\Omega )\times
 H^{\frac{1}{2}}({\partial\Omega})$.  
 Condition \eqref{3.suf}
is necessary and sufficient for solvability of the 
nonhomogeneous BDIE systems {\rm(N1)} and
 {\rm(N2)} in $H^1(\Omega )\times
 H^{\frac{1}{2}}({\partial\Omega})$.
\end{theorem} 
\begin{proof} 
(i) Let $u\in H^1(\Omega )$ be a solution of the
Neumann  problem
 \eqref{2.6N}-\eqref{2.8}. Then from Theorem~\ref{GIequivalenceH-1} and relations \eqref{4.2u+Gen} and \eqref{4.2T+Gen} we see
that the couple $(u,\varphi)$ with $\varphi= \gamma^+u$ solves LBDIE systems (N1)
and (N2), which proves item (i).

(ii) Let a couple $(u,\varphi)\in H^1(\Omega )\times
H^{\frac{1}{2}}({\partial\Omega})$ solve LBDIE system (N1). 
Lemma \ref{IDequivalenceGen} for equation \eqref{Basic1} implies that
$u$ is a solution of equation \eqref{2.1}, and equations
\eqref{difference}-\eqref{4.2T+gen} hold for $\Psi=\psi_0$ and
$\Phi=\varphi$. 
Subtracting \eqref{4.2T+gen} from equation \eqref{Basic2} gives $T^+ (\tilde f,u)=\psi_0$ on $\partial\Omega$. 
Further, from  \eqref{difference} we derive
$ W(\gamma^+u - \varphi)=0$  in $\Omega^{+},$
whence  $\gamma^+u = \varphi$  on $\partial \Omega$ by Lemma
 \ref{VW0s}, completing item (ii) for LBDIE system (N1).

Let now a couple $(u,\varphi)\in H^1(\Omega )\times
H^{\frac{1}{2}}({\partial\Omega})$ solve the LBDIE system
(N2). 
Further,
taking the trace of \eqref{Basic1-2} on ${\partial\Omega}$ and
comparing the result with \eqref{Basic2-2}, we easily derive that
$\gamma^+u=\varphi$ on ${\partial\Omega}$. 
Lemma \ref{IDequivalenceGen} for equation \eqref{Basic1-2} implies that
$u$ is a solution of equation \eqref{2.1}, while equations
\eqref{difference}-\eqref{4.2T+gen} hold for $\Psi=\psi_0$ and
$\Phi=\varphi$.
Further, from  \eqref{difference} we derive
$$
 V(\psi_0 -T^+(\tilde f,u) )=0 \quad   \mbox{\rm in}\;\;\;\;
  \Omega^{+},
$$
whence $T^+ u=\psi_0$ on ${\partial\Omega}$ by Lemma
 \ref{VW0s}, i.e., $u$ solves the
Neumann problem
 \eqref{2.6N}-\eqref{2.8}, which completes the proof of item (ii) for LBDIE system (N2).

(iii)
 Theorem~\ref{Rem1N} along with items (i) and (ii)  imply the
claims of item (iii) for LBDIE systems (N2) and  (N1).
\end{proof} 
\subsection{Properties of BDIE system operators for the Neumann problem}

BDIE systems (N1) and (N2) can be written, respectively, as
\begin{equation*} 
{\mathfrak N}^1\mathcal U^N=\mathcal F^{N1},\quad
{\mathfrak N}^2\mathcal U^N=\mathcal F^{N2},
\end{equation*}
where 
$
{\mathcal U}^N=(u, \varphi)^\top
\in H^1(\Omega)\times {H}^{\frac{1}{2}}(\partial_D\Omega),
$
\begin{equation*}
{\mathfrak N}^1:= \left[
\begin{array}{cc}
I+{\mathcal R} &\ W  \\[1ex]
 T^+{\mathcal R}  &\ {\mathcal L}^+
\end{array}\right],\quad
{\mathfrak N}^2:= \left[
\begin{array}{cc}
I+{\mathcal R} & W  \\[1ex]
\,\gamma^+ {\mathcal R}& \quad \displaystyle \frac{1}{2}I + {\mathcal W}
\end{array}
\right].
\end{equation*}
Due to the mapping properties of the potentials,
$\mathcal F^{N1}\in   H^1 (\Omega)\times H^{-\ha}(\pO),$
$\mathcal F^{N2}\in  H^1 (\Omega)\times H^{\ha}(\pO).$

\begin{theorem}\label{invertibilityN}
 The operators 
\begin{eqnarray}
  \mathfrak N^1&:& H^1(\Omega )\times
H^{\frac{1}{2}}(\partial\Omega) \to H^1(\Omega )\times
H^{-\frac{1}{2}}(\partial\Omega),\label {mfN1}\\
\label{mfN2}
\mathfrak N^2&:&H^1(\Omega )\times H^{\frac{1}{2}}({\partial\Omega})\to
  H^1(\Omega )\times H^{\frac{1}{2}}({\partial\Omega}).
  \end{eqnarray}
are continuous Fredholm operators with zero index. They have one--dimensional null--spaces,
$\ker \mathfrak N^1=\ker \mathfrak N^2$, in $H^1(\Omega )\times
H^{\frac{1}{2}}({\partial\Omega})$, spanned over the element $(u^0,\varphi^0)=(1,1)$.
\end{theorem}
\begin{proof} 
The mapping properties of the potentials, see Appendix, imply continuity of operators \eqref{mfN2}  and \eqref{mfN1}.

First consider operator \eqref{mfN1}. Let us denote $\L_0^+g:={\L}^+_{\Delta}(ag)$. Hence the operator $\L_0^+: H^{\frac{1}{2}}(\partial\Omega) \to
H^{-\frac{1}{2}}(\partial\Omega)$ is a Fredholm operator with zero index (cf. e.g. \cite[Theorem 2]{Costabel1988}, \cite[Ch. XI, Part B,
\S 3,]{DaLi4}).
Therefore the operator
\begin{equation}
\label{d27}
{\cal A}^{N1}_0:=
 \left[
 \begin{array}{cc}
I     &   W \\
 0      & \L_0^+
  \end{array}
 \right]  \,:\, H^1(\Omega )\times
H^{\frac{1}{2}}(\partial\Omega) \to H^1(\Omega )\times
H^{-\frac{1}{2}}(\partial\Omega).
  \end{equation}
is also Fredholm with zero index.
Operator \eqref{mfN1} is a compact perturbation of ${\cal A}^{N1}_0$ since the operators
$$
\begin{array}{l}
{\cal R}\, :\, H^{1}(\Omega)\to H^{1}(\Omega)\\
\L^+-\L_0^+\, :\, H^{\frac{1}{2}}(\partial\Omega) \to
H^{-\frac{1}{2}}(\partial\Omega),\\
 T^+{\cal R} \, :\, H^1(\Omega )\to H^{-\frac{1}{2}}(\partial\Omega)
\end{array}
$$
are compact, due to relation \eqref{VWab4} and Theorem \ref{T3.3}.
Thus operator  \eqref{mfN1} is Fredholm with zero index as well. The
claims that $\ker \mathfrak N^1$ is one--dimensional and the
couple $(u^0,\varphi^0)=(1,1)$ belongs to $\ker \mathfrak N^1$
directly follow from Theorem \ref{equivalenceN}(iii).

The proof for operator \eqref{mfN2} is similar.
\end{proof}

To describe in more details the ranges of operators \eqref{mfN1} and  \eqref{mfN2}, i.e., to give more information about the co-kernels of these operators,  we will 
need several auxiliary assertions. 
First of all, let us remark that for any $v\in  H^{s-\tha}(\partial\Omega)$, $s<\frac{3}{2}$, the single layer potential can be defined as
\begin{align}\label{V-P}
Vv(y):=-\langle \gamma P(\cdot,y), v\rangle_{\partial\Omega}
=-\langle P(\cdot,y),\gamma^*  v\rangle_{\R^3}
=-\mathbf P \,\gamma^*v(y),\quad y\in \R^3\setminus\partial\Omega.
\end{align}
where $\gamma^*: H^{s-\tha}(\partial\Omega)\to H^{s-2}_{\partial\Omega}$, $s<\frac{3}{2}$, is the operator adjoined to the trace operator $\gamma:H^{2-s}(\R^3)\to H^{\tha-s}(\partial\Omega)$, and the space $H^{s-2}_{\partial\Omega}$ is defined by \eqref{H_dO}.

\begin{lemma}\label{lemma3P1}
Let  $\tilde f\in
  \widetilde H^{s-2}(\Omega)$, $s>\ha$.  If
\begin{equation}
 \label{B.190}
r_\Omega\mathbf P \,\tilde f=0\quad \text{in }\Omega,
\end{equation}
then $\tilde f=0$ in $\R^3$.
\end{lemma}
\begin{proof} 
Multiplying \eqref{B.190} by $a$, taking into account \eqref{3.d1}  and applying the Laplace operator, we obtain $r_\Omega\tilde f=0$, which means $\tilde f\in H^{s-2}_{\partial\Omega}$. If $s\ge\tha$, then $\tilde f=0$ by Theorem 2.10 from \cite{MikJMAA2011}. If $\ha<s<\tha$, then by the same theorem  there exists $v\in H^{s-\tha}(\partial\Omega)$ such that $\tilde f=\gamma^*v$. This gives
$\mathbf P \,\tilde f=\mathbf P \,\gamma^*v=-Vv$ in $\R^3$.
Then \eqref{B.190} reduces to
$Vv=0$ in $\Omega$, which implies  $v=0$ on $\partial\Omega$
 (see e.g. Lemma \ref{VW0s} for $s=1$, which can be easily generalized to $\ha<s<\tha$) and thus $\tilde f=0$ in $\R^3$.
\end{proof} 

\begin{theorem}\label{teoremP1inv}
Let $\ha<s<\tha$. The operator
\begin{equation}
 \label{B.19wP1}
\mathbf P:\widetilde H^{s-2}(\Omega )\to H^{s}(\Omega )
\end{equation}
and its inverse
\begin{equation}
 \label{B.19wP1in}
(\mathbf P)^{-1}: H^{s}(\Omega )\to \widetilde H^{s-2}(\Omega )
\end{equation}
are continuous and
\be\label{B.194a}
(\mathbf P)^{-1} g
=[\Delta \mathring E (I-  V_\Delta \mathcal{V}_{\Delta}^{-1}\gamma^+)
-\gamma^*\mathcal{V}_{\Delta}^{-1}\gamma^+](ag)\ \mbox{ in } \R^3,\quad \forall g\in  H^{s}(\Omega ).
\ee
\end{theorem}
\begin{proof} 
The continuity of \eqref{B.19wP1} is well known, cf. \cite[Theorem 3.8]{CMN-1}. By Lemma~\ref{lemma3P1}, operator \eqref{B.19wP1} is injective. Let us prove its surjectivity. To this end, for arbitrary $g\in H^{s}(\Omega )$ let us consider the following equation with respect to $\tilde f\in \widetilde H^{s-2}(\Omega )$,
\begin{equation}
 \label{B.191}
\mathbf P_\Delta  \,\tilde f=g \text{ in }\Omega .
\end{equation}

Let $g_1\in H^{s}(\Omega )$ be the (unique) solution of the following Dirichlet problem: $\Delta g_1=0$ in $\Omega$, $\gamma^+g_1=\gamma^+g$, which can be particularly presented as $g_1=V_\Delta \mathcal{V}_{\Delta}^{-1}\gamma^+g,$ see e.g \cite{Costabel1988} or proof of Lemma 2.6 in \cite{MikJMAA2011}.
Let $g_0:=g- g_1$. Then $g_0\in H^s(\Omega)$ and $\gamma^+g_0=0$ and thus $g_0$ can be uniquely extended  to $\mathring E g_0\in \widetilde H^s(\Omega)$, where $\mathring E$ is the operator of extension by $0$ outside $\Omega$. 
Thus by  \eqref{V-P}, equation \eqref{B.191} takes form 
\begin{equation}
 \label{B.192}
r_\Omega\mathbf P_\Delta  [\tilde f+\gamma^*\mathcal{V}_{\Delta}^{-1}\gamma^+g]=g_0 \text{ in }\Omega.
\end{equation}
Any solution $\tilde f\in \widetilde H^{s-2}(\Omega )$ of the corresponding equation on $\R^3$,
\begin{equation}
 \label{B.193}
\mathbf P_\Delta  [\tilde f+\gamma^*\mathcal{V}_{\Delta}^{-1}\gamma^+g]= \mathring E g_0\ \mbox{ in } \R^3,
\end{equation}
will evidently solve \eqref{B.192}. If $\tilde f$ solves \eqref{B.193} then acting with the Laplace operator on \eqref{B.193}, we obtain
\be\label{B.194}
\tilde f=\tilde Q g:= \Delta \mathring E g_0 -\gamma^*\mathcal{V}_{\Delta}^{-1}\gamma^+g
=\Delta \mathring E (g- r_\Omega V_\Delta \mathcal{V}_{\Delta}^{-1}\gamma^+g)
-\gamma^*\mathcal{V}_{\Delta}^{-1}\gamma^+g\ \mbox{ in } \R^3.
\ee
On the other hand, substituting $\tilde f$ given by \eqref{B.194} to \eqref{B.193} and taking into account that $\mathbf P_\Delta  \Delta\tilde h=\tilde h$ for any $\tilde h\in \widetilde H^{s}(\Omega )$, $s\in\R$, we obtain that $\tilde Q g$ is indeed a solution of equation \eqref{B.193} and thus \eqref{B.192}. 
By Lemma~\ref{lemma3P1} the solution of \eqref{B.192} is unique, which means that the operator $\tilde Q$ is inverse to operator \eqref{B.19wP1}, i.e., $\tilde Q=(r_\Omega\mathbf P)^{-1}$.
Since $\Delta$ is a continuous operator from $\widetilde H^{s}(\Omega )$ to $\widetilde H^{s-2}(\Omega )$, equation \eqref{B.194} implies that the operator $(r_\Omega\mathbf P)^{-1}=\tilde Q:H^{s}(\Omega )\to \widetilde H^{s-2}(\Omega )$ is continuous.
The relations $\mathbf P =\frac{1}{a}\mathbf P_\Delta $ and $a(x)>c>0$ then imply invertibility of operator \eqref{B.19wP1} and ansatz \eqref{B.194a}.
\end{proof} 

\begin{lemma}\label{L4}
{
For any couple $(\F_1,\F_2)\in H^{1}(\Omega)\times {H}^{-\ha}(\pO)$, 
there exists a unique couple $(\tilde f_{**}, \Phi_*)\in \s{H}^{-1} (\Omega)\times 
{H}^{\ha}(\pO)$
 such that
\begin{eqnarray}
\label{4.8Psi1**} \F_1 &=&{\cal P}\tilde f_{**}-W\Phi_* \quad
\text{in }\Omega,\\
\label{4.8Psi2**}
\F_2&=&T^+(\tilde f_{**}+\mathring E\,{\mathcal R}_*\tilde f_{**},{\cal P}\tilde f_{**})-\L^+ \Phi_* 
  \quad
\text{on }\partial\Omega.
\end{eqnarray}
Moreover, $(\tilde f_{**}, \Phi_*)={\cal C}_{**}(\F_1,\F_2)$ and
${\cal C}_{**}: H^{1}(\Omega)\times {H}^{-\ha}(\pO)\to \s{H}^{-1} (\Omega)\times {H}^{\ha}(\pO)$ 
is a linear continuous operator given by 
\begin{align}\label{f**1}
\tilde f_{**}=&\check{\Delta}(a\F_1)+\gamma^*(\F_2+(\gamma^+\F_1)\partial_n a),\\
\label{Phi*F**0}
 \Phi_* 
=&\frac{1}{a}\left(-\ha I + {\cal W}_\Delta\right)^{-1}\gamma^+
\{ -a\F_1+{\cal P}_\Delta[\check{\Delta}(a\F_1)+\gamma^*(\F_2+(\gamma^+\F_1)\partial_n a)]\},
\end{align}
where $\check{\Delta}(a\F_1)=\nabla\cdot\mathring E \nabla (a\F_1)$.
}
\comment{
(ii) For any couple $(\F_1,\F_2)\in H^{1}(\Omega)\times {H}^{-\ha}(\pO)$,
there exists a triple
$(\tilde f_*,\Phi_*,\Psi_*)
={\cal C}_*(\F_1,\F_2)
\in \s{H}^{-1} (\Omega)\times {H}^{\ha}(\pO)\times {H}^{-\ha}(\pO)$
 such that
\begin{eqnarray}
\label{4.8Psi1} \F_1 &=&{\cal P}\tilde f_* + V\Psi_* - W\Phi_* \quad
\text{in }\Omega,\\
\label{4.8Psi2}
\F_2&=&T^+(\tilde f_*+\mathring E\,{\mathcal R}_*\tilde f_*,{\cal P}\tilde f_*)  
-\frac{1}{2}\Psi_*+{\Wp}\Psi_*
-\L^+ \Phi_*
  \quad
\text{on }\partial\Omega.
\end{eqnarray}
and 
${\cal C}_*: H^{1}(\Omega)\times {H}^{-\ha}(\pO)\to 
\s{H}^{-1} (\Omega)\times {H}^{\ha}(\pO)\times {H}^{-\ha}(\pO)$ 
is a linear continuous operator. 
}
\end{lemma}
\begin{proof}
Let us first assume that there exist $(\tilde f_{**}, \Phi_*)\in \s{H}^{-1} (\Omega)\times {H}^{\ha}(\pO)$
satisfying equations \eqref{4.8Psi1**}, \eqref{4.8Psi2**} and find their expressions in terms of $\F_1$ and $\F_2$.
Let us re-write \eqref{4.8Psi1**} as 
\begin{align}
\label{4.8Psi1a**} \F_1-{\cal P}\tilde f_{**} = - W\Phi_* \quad
\text{in }\Omega,
\end{align}
Multiplying \eqref{4.8Psi1a**} by $a$ and applying the Laplace operator to it, we obtain, 
\begin{align}
\label{2.6'genPhi=**}
& \Delta(a\F_1-\P_\Delta \tilde f_{**})=\Delta(a\F_1)- \tilde f_{**}=-\Delta W_\Delta(a\Phi_*)=0 
\text{in}  \Omega,
\end{align}
which means
\begin{align}
\label{2.6'genPhi**}
& \Delta(a\F_1)=r_\Omega \tilde f_{**} \text{in}  \Omega
\end{align}
and
$a\F_1-\P_\Delta \tilde f_{**}\in H^{1,0}(\Omega;\Delta)$ and hence $\F_1-\P \tilde f_{**}\in H^{1,0} (\Omega;A)$.
The latter implies that the canonical conormal derivative $T^+(\F_1-\P \tilde f_{**})$ is well defined. 
It can be also written in terms of the generalised 
conormal derivatives,
\begin{multline}\label{6.31**}
T^+(\F_1-\P \tilde f_{**})
=T^+(\tilde A(\F_1-\P\tilde f_{**}),\F_1-\P \tilde f_{**})
=T^+(\mathring E A(\F_1-\P\tilde f_{**}),\F_1-\P \tilde f_{**})\\
=T^+(\mathring E \nabla\cdot (a\nabla(\F_1-\P\tilde f_{**})),\F_1-\P \tilde f_{**})\\
=T^+(\mathring E \Delta(a\F_1-\P_\Delta\tilde f_{**})
-\mathring E \nabla\cdot ((\F_1-\P\tilde f_{**})\nabla a),\F_1-\P \tilde f_{**})\\
=T^+(-\mathring E\nabla\cdot(\F_1\nabla a)
-\mathring E \mathcal R_* \tilde f_{**},\F_1-\P \tilde f_{**}),
\end{multline}
\comment{
\begin{multline}\label{6.31**}
T^+(\F_1-\P \tilde f_{**})
=(\gamma^{-1})^*[\tilde A(\F_1-\P\tilde f_{**})-\check A(\F_1-\P \tilde f_{**})]
=(\gamma^{-1})^*[\mathring E (A\F_1-r_\Omega \tilde f_{**}-\mathcal R_* \tilde f_{**})-\check A(\F_1)
+\tilde f_{**}-\tilde f_{**}+\check A\P \tilde f_{**}]\\
=(\gamma^{-1})^*[-\mathring E\nabla\cdot(\F_1\nabla a)+ \mathring E (\Delta(a\F_1)-r_\Omega \tilde 
f_{**})+\tilde 
f_{**}-\check A(\F_1)]
-T^+(\tilde f_{**}+\mathring E \mathcal R_* \tilde f_{**},\P \tilde f_{**})\\
=T^+(\tilde f_{**}-\mathring E\nabla\cdot(\F_1\nabla a),\F_1)
-T^+(\tilde f_{**}+\mathring E \mathcal R_* \tilde f_{**},\P \tilde f_{**}),
\end{multline}
}
where \eqref{2.6'genPhi**} and \eqref{R^0} were taken into account.
Applying the canonical conormal derivative operator $T^+$ to the both sides of 
equation \eqref{4.8Psi1a**} and substituting there \eqref{6.31**}, we obtain
\begin{align}
\label{6.32**}
T^+(\tilde f_{**}-\mathring E\nabla\cdot(\F_1\nabla a),\F_1)-T^+(\tilde 
f_{**}+\mathring E \mathcal R_* \tilde f_{**},\P \tilde f_{**}) 
=& -\L^+ \Phi_*
  \quad
\text{on }\partial\Omega.
\end{align}
Subtracting this from \eqref{4.8Psi2**}, we obtain
\begin{eqnarray}
\label{4.8Psi2T**}
\F_2&=&T^+(\tilde f_{**}-\mathring E\nabla\cdot(\F_1\nabla a),\F_1)
  \quad
\text{on }\partial\Omega.
\end{eqnarray}

Due to \eqref{2.6'genPhi**}, we can represent
\begin{align}\label{f**}
\tilde f_{**}=\check{\Delta}(a\F_1)+\tilde f_{1*}=\nabla\cdot\mathring E \nabla (a\F_1)-\gamma^*\Psi_{**},
\end{align}
where $\tilde f_{1*}\in H^{-1}_{\partial\Omega}$ and hence, due to e.g. \cite[Theorem 2.10]{MikJMAA2011} can be 
in turn represented as $\tilde f_{1*}=-\gamma^*\Psi_{**}$, with some $ \Psi_{**}\in H^{-\ha}(\partial\Omega)$.
Then \eqref{2.6'genPhi**} is satisfied and 
\begin{multline}
\label{4.8Psi2TT**}
T^+(\tilde f_{**}-\mathring E\nabla\cdot(\F_1\nabla a),\F_1)
=(\gamma^{-1})^*[\tilde f_{**}-\mathring E\nabla\cdot(\F_1\nabla a)-\check 
A\F_1]\\
=(\gamma^{-1})^*[\nabla\cdot\mathring E \nabla (a\F_1) -\gamma^*\Psi_{**}
-\mathring E\nabla\cdot(\F_1\nabla a)-\nabla\cdot\mathring E (a\nabla \F_1)]\\
=(\gamma^{-1})^*[\nabla\cdot\mathring E (\F_1\nabla a)-\gamma^*\Psi_{**}-\mathring E\nabla\cdot(\F_1\nabla a)]
=-\Psi_{**}-(\gamma^+\F_1)\partial_na
\end{multline}
because
\begin{multline*}
\left\langle(\gamma^{-1})^*[\nabla\cdot\mathring E (\F_1\nabla a)-\gamma^*\Psi_{**}
-\mathring E\nabla\cdot(\F_1\nabla a)], w\right\rangle_{\partial\Omega}
=\left\langle\nabla\cdot\mathring E (\F_1\nabla a)-\gamma^*\Psi_{**}
-\mathring E\nabla\cdot(\F_1\nabla a),\gamma^{-1}w\right\rangle_{\Omega}\\
=\left\langle\nabla\cdot\mathring E (\F_1\nabla a),\gamma^{-1}w\right\rangle_{\R^3}-\Psi_{**}
-\left\langle\mathring E\nabla\cdot(\F_1\nabla a),\gamma^{-1}w\right\rangle_{\Omega}\\
=-\left\langle\mathring E (\F_1\nabla a),\nabla\gamma^{-1}w\right\rangle_{\R^3}-\Psi_{**}
+\left\langle\F_1\nabla a,\nabla\gamma^{-1}w\right\rangle_{\Omega}
-\left\langle n\cdot\gamma^+(\F_1\nabla a),\gamma^+\gamma^{-1}w\right\rangle_{\Omega}
=-\left\langle(\gamma^+\F_1)\partial_na,w\right\rangle_{\partial\Omega}-\Psi_{**}.
\end{multline*}
Hence \eqref{4.8Psi2T**} reduces to
\begin{eqnarray}
\label{4.8Psi2T3}
\Psi_{**}=-\F_2-(\gamma^+\F_1)\partial_n a .
\end{eqnarray}
and \eqref{f**} to \eqref{f**1}.

Now \eqref{4.8Psi1a**} can be written in the form
\begin{eqnarray}
W_\Delta (a\Phi_*)  = \F_\Delta \quad
\text{in }\Omega,
\label{WF**}
\end{eqnarray}  
where 
\begin{eqnarray}
\label{4.8Psi1c**} \F_\Delta:=-a\F_1+{\cal P}_\Delta\tilde f_{**}
= -a\F_1+{\cal P}_\Delta[\check{\Delta}(a\F_1)+\gamma^*(\F_2+(\gamma^+\F_1)\partial_n a)], \quad
\end{eqnarray}  
is a harmonic function in $\Omega$ due to \eqref{2.6'genPhi=**}.
The trace of equation \eqref{WF**} gives 
\begin{align}\label{Phi*eq**}
 -\ha a\Phi_* + {\cal W}_\Delta(a\Phi_*)=\gamma^+\F_\Delta\  \mbox{ on } \partial\Omega.
\end{align}
Since the operator 
$ -\ha I + {\cal W}_\Delta : H^{\frac{1}{2}}(\partial\Omega) \to H^{-\frac{1}{2}}(\partial\Omega)$  is an 
isomorphism (see
e.g. \cite[Ch. XI, Part B, \S 2, Remark 8]{DaLi4}), this implies 
\begin{align}\label{Phi*F**}
 \Phi_* =&\frac{1}{a}\left(-\ha I + {\cal W}_\Delta\right)^{-1}\gamma^+\F_\Delta\nonumber\\
=&\frac{1}{a}\left(-\ha I + {\cal W}_\Delta\right)^{-1}\gamma^+
\{ -a\F_1+{\cal P}_\Delta[\check{\Delta}(a\F_1)+\gamma^*(\F_2+(\gamma^+\F_1)\partial_n a)]\}.
\end{align}

Relations \eqref{f**1},  \eqref{Phi*F**} can be written as 
$(\tilde f_*,\Phi_{*})={\cal C}_{**}(\F_1,\F_2)$,
where 
${\cal C}_{**}: H^{1}(\Omega)\times {H}^{-\ha}(\pO)\to \s{H}^{-1} (\Omega)\times {H}^{\ha}(\pO)$
is a linear continuous operator, as requested.
We still have to check that the functions $\tilde f_{**}$ and $\Phi_{*}$, given by \eqref{f**1} and  \eqref{Phi*F**}, 
satisfy equations \eqref{4.8Psi1**} and \eqref{4.8Psi2**}.
Indeed, $\Phi_{*}$ given by \eqref{Phi*F**} satisfies equation \eqref{Phi*eq**} and thus 
$\gamma^+W_\Delta (a\Phi_{*})  = \gamma^+\F_\Delta$.
Since both $W_\Delta (a\Phi_{*}) $ and  $\F_\Delta$ are harmonic functions, this implies \eqref{WF**}-\eqref{4.8Psi1c**} 
and by \eqref{f**1} also \eqref{4.8Psi1**}.
Finally, \eqref{f**1} implies by \eqref{4.8Psi2TT**} that \eqref{4.8Psi2T**} is satisfied, and adding \eqref{6.32**} to it 
leads to \eqref{4.8Psi2**}.

Let us now prove that the operator ${\cal C}_{**}$ is unique.
Indeed, let a couple
$(\tilde f_*,\Phi_*)\in \s{H}^{-1} (\Omega)\times {H}^{\ha}(\pO)$ 
be a solution of linear system \eqref{4.8Psi1**}-\eqref{4.8Psi2**} with $\F_1=0$ and $\F_2=0$. 
Then \eqref{2.6'genPhi**} implies that $r_\Omega \tilde f_{**}=0$ in $\Omega$, i.e., 
$\tilde f_{**}\in{H}^{-1}_{\partial \Omega}\subset \s{H}^{-1} (\Omega)$.
Hence, \eqref{4.8Psi2T**} reduces to
\begin{eqnarray}
\label{4.8Psi2T0**}
0&=&T^+(\tilde f_{**},0)
  \quad
\text{on }\partial\Omega.
\end{eqnarray}
By the first Green identity \eqref{Tgen}, this gives,
\begin{equation}
\label{Tgen*} 
0=\left\langle T^+(\tilde f_{**},0), \gamma^{+}v \right\rangle _{\pO}
=\langle \tilde f_{**},v \rangle_\Omega \quad  \forall\ v\in H^1 (\Omega),
\end{equation}
which implies $\tilde f_{**}=\gamma^*\Psi_*$. 
Finally, \eqref{Phi*F**} gives $\Phi_*=0$. 
Hence, any solution of non-homogeneous linear system \eqref{4.8Psi1**}-\eqref{4.8Psi2**} has only one solution, which implies uniqueness of the operator ${\cal C}_{**}$. 

\comment{
{\rd(ii)} Let us first assume that there exist $(\tilde f_{*}, \Phi_*, \Psi_*)\in \s{H}^{-1} (\Omega)\times 
{H}^{\ha}(\pO)\times 
{H}^{-\ha}(\pO)$
satisfying equations \eqref{4.8Psi1}, \eqref{4.8Psi2} and find their expressions 
in terms of $\F_1$ and $\F_2$.
Let us re-write \eqref{4.8Psi1} as 
\begin{align}
\label{4.8Psi1a} \F_1-{\cal P}\tilde f_* = V\Psi_* - W\Phi_* \quad
\text{in }\Omega,
\end{align}
Multiplying it by $a$ and applying the Laplace operator to it, we obtain, 
\begin{align}
\label{2.6'genPhi=}
& \Delta(a\F_1-\P_\Delta \tilde f_{*})=\Delta(a\F_1)- \tilde f_{*}=\Delta(V_\Delta\Psi_* - W_\Delta(a\Phi_*))=0 \text{in}  \Omega,
\end{align}
which means
\begin{align}
\label{2.6'genPhi}
& \Delta(a\F_1)=r_\Omega \tilde f_{*} \text{in}  \Omega
\end{align}
and
$a\F_1-\P_\Delta \tilde f_{*}\in H^{1,0}(\Omega;\Delta)$ and hence $\F_1-\P \tilde f_{*}\in H^{1,0} (\Omega;A)$.
The latter implies that the canonical conormal derivative $T^+(a\F_1-\P_\Delta 
\tilde f_{*})$ is well defined. It can be also written in terms of the generalised 
conormal derivatives,
\begin{multline}\label{6.31}
T^+(\F_1-\P \tilde f_{*})
=(\gamma^{-1})^*[\tilde A(\F_1-\P\tilde f_{*})-\check A(\F_1-\P \tilde f_{*})]
=(\gamma^{-1})^*[\mathring E A(\F_1-\P \tilde f_{*})-\check A(\F_1)+\check A\P \tilde f_{*}]\\
=(\gamma^{-1})^*[\mathring E (A\F_1-r_\Omega \tilde f_{*}-\mathcal R_* \tilde f_{*})-\check A(\F_1)
+\tilde f_{*}-\tilde f_{*}+\check A\P \tilde f_{*}]\\
=(\gamma^{-1})^*[-\mathring E\nabla\cdot(\F_1\nabla a)+ \mathring E (\Delta(a\F_1)-r_\Omega \tilde f_{*})+\tilde f_{*}-\check A(\F_1)]
-T^+(\tilde f_{*}+\mathring E \mathcal R_* \tilde f_{*},\P \tilde f_{*})\\
=T^+(\tilde f_{*}-\mathring E\nabla\cdot(\F_1\nabla a),\F_1)
-T^+(\tilde f_{*}+\mathring E \mathcal R_* \tilde f_{*},\P \tilde f_{*})
\end{multline}
since $\Delta(a\F_1)-r_\Omega \tilde f_{*}=0$ by \eqref{2.6'genPhi}.
Applying the canonical conormal derivative operator $T^+$ to the both sides of 
equation \eqref{4.8Psi1a} and substituting there \eqref{6.31}, we obtain
\begin{align}
\label{6.32}
T^+(\tilde f_{*}-\mathring E\nabla\cdot(\F_1\nabla a),\F_1)-T^+(\tilde 
f_{*}+\mathring E \mathcal R_* \tilde f_{*},\P \tilde f_{*}) 
=&\dfrac{1}{2}\Psi_*+{\Wp}\Psi_* -\L^+ \Phi_*
  \quad
\text{on }\partial\Omega.
\end{align}
Subtracting this from \eqref{4.8Psi2}, we obtain
\begin{eqnarray}
\label{4.8Psi2T}
\F_2&=&T^+(\tilde f_{*}-\mathring E\nabla\cdot(\F_1\nabla a),\F_1)-\Psi_*
  \quad
\text{on }\partial\Omega.
\end{eqnarray}
Let us take 
\begin{align}\label{f*}
\tilde f_{*}=\check{\Delta}(a\F_1)=\nabla\cdot\mathring E \nabla (a\F_1).
\end{align}
Then \eqref{2.6'genPhi} is satisfied and 
\begin{multline}
\label{4.8Psi2TT}
T^+(\tilde f_{*}-\mathring E\nabla\cdot(\F_1\nabla a),\F_1)
=(\gamma^{-1})^*[\tilde f_{*}-\mathring E\nabla\cdot(\F_1\nabla a)-\check 
A\F_1]\\
=(\gamma^{-1})^*[\nabla\cdot\mathring E \nabla (a\F_1)-\mathring 
E\nabla\cdot(\F_1\nabla a)-\nabla\cdot\mathring E (a\nabla \F_1)]\\
=(\gamma^{-1})^*[\nabla\cdot\mathring E (\F_1\nabla a)-\mathring 
E\nabla\cdot(\F_1\nabla a)]
=-(\gamma^+\F_1)\partial_na
\end{multline}
because
\begin{multline*}
\left\langle(\gamma^{-1})^*[\nabla\cdot\mathring E (\F_1\nabla a)-\mathring 
E\nabla\cdot(\F_1\nabla a),w\right\rangle_{\partial\Omega}
=\left\langle\nabla\cdot\mathring E (\F_1\nabla a)-\mathring 
E\nabla\cdot(\F_1\nabla a),\gamma^{-1}w\right\rangle_{\Omega}\\
=\left\langle\nabla\cdot\mathring E (\F_1\nabla a),\gamma^{-1}w\right\rangle_{\R^3}
-\left\langle\mathring E\nabla\cdot(\F_1\nabla a),\gamma^{-1}w\right\rangle_{\Omega}\\
=-\left\langle\mathring E (\F_1\nabla a),\nabla\gamma^{-1}w\right\rangle_{\R^3}
+\left\langle\F_1\nabla a,\nabla\gamma^{-1}w\right\rangle_{\Omega}
-\left\langle n\cdot\gamma^+(\F_1\nabla a),\gamma^+\gamma^{-1}w\right\rangle_{\Omega}
=-\left\langle(\gamma^+\F_1)\partial_na,w\right\rangle_{\partial\Omega}.
\end{multline*}
Hence \eqref{4.8Psi2T} reduces to
\begin{eqnarray}
\label{4.8Psi2T3}
\Psi_*=-\F_2-(\gamma^+\F_1)\partial_n a .
\end{eqnarray}

Now \eqref{4.8Psi1a} can be written in the form
\begin{eqnarray}
W_\Delta (a\Phi_*)  = \F_\Delta \quad
\text{in }\Omega,
\label{WF}
\end{eqnarray}  
where 
\begin{eqnarray}
\label{4.8Psi1c} \F_\Delta:=V_\Delta\Psi_*-a\F_1+{\cal P}_\Delta\tilde f_*
= -V_\Delta(\F_2+(\gamma^+\F_1)\partial_n a)-a\F_1+{\cal P}_\Delta\check{\Delta}(a\F_1), \quad
\end{eqnarray}  
is a harmonic function in $\Omega$ due to \eqref{2.6'genPhi=}.
The trace of equation \eqref{WF} gives 
\begin{align}\label{Phi*eq}
 -\ha a\Phi_* + {\cal W}_\Delta(a\Phi_*)=\gamma^+\F_\Delta\  \mbox{ on } \partial\Omega.
\end{align}
Since the operator 
$ -\ha I + {\cal W}_\Delta : H^{\frac{1}{2}}(\partial\Omega) \to H^{-\frac{1}{2}}(\partial\Omega)$  is an 
isomorphism (see
e.g. \cite[Ch. XI, Part B, \S 2, Remark 8]{DaLi4}), this implies 
\begin{align}\label{Phi*F}
 \Phi_* =&\frac{1}{a}\left(-\ha I + {\cal W}_\Delta\right)^{-1}\gamma^+\F_\Delta\nonumber\\
=&\frac{1}{a}\left(-\ha I + {\cal W}_\Delta\right)^{-1}\gamma^+
[ -V_\Delta(\F_2+(\gamma^+\F_1)\partial_n a)-a\F_1+{\cal P}_\Delta\check{\Delta}(a\F_1)].
\end{align}
\comment%
{
{\rd *** Old proof:}

Introducing $\tilde f_{**}=\tilde f_{*}-\gamma^*\Psi_*$ and taking into account \eqref{V-P},
equation \eqref{4.8Psi1} can be rewritten as
\begin{eqnarray}
\label{4.8Psi1**}  {\mathbf P}\tilde f_{**}={\cal P}\tilde f_{**}&=&\F_1  \quad
\text{in }\Omega.
\end{eqnarray}
On the other hand, since 
$$
\dfrac{1}{2}\Psi_*+{\Wp}\Psi_*=T^+ V\Psi_*=T^+(\mathring E\, A V\Psi_* ;V\Psi_*)
=T^+(-\mathring E\, {\mathcal R}_*\gamma^*\Psi_* ;-\mathcal P\gamma^*\Psi_*)
$$
and by the co-normal derivative linearity \eqref{GCDL}, equation \eqref{4.8Psi2} can be rewritten as
\begin{eqnarray}
\label{4.8Psi2**}
T^+({\rd \tilde f_{*}}+\mathring E\,{\mathcal R}_*\tilde f_{**},{\cal P}\tilde f_{**})-\Psi_*&=&\F_2  
\quad
\text{on }\partial\Omega.
\end{eqnarray}
By Theorem~\ref{teoremP1inv}, equation \eqref{4.8Psi1**} has the unique solution
\be\label{f**s}
\tilde f_{**}=(\mathbf P)^{-1}\F_1,
\ee 
equation \eqref{4.8Psi2**} gives
\begin{eqnarray}
\label{4.8Psi2**s}
\Psi_*&=&T^+(\tilde f_{**}+\mathring E\,{\mathcal R}_*\tilde f_{**},{\cal P}\tilde f_{**})-\F_2  \quad
\text{on }\partial\Omega,
\end{eqnarray}
and finally,
\be\label{f*s}
\tilde f_{*}=\tilde f_{**}+\gamma^*\Psi_*.
\ee
}

Relations \eqref{f*}, \eqref{4.8Psi2T3}, \eqref{Phi*F} can be written as 
$(\tilde f_*,\Phi_*,\Psi_*)
={\cal C}_*(\F_1,\F_2)$,
where 
${\cal C}_*: H^{1}(\Omega)\times {H}^{-\ha}(\pO)\to 
\s{H}^{-1} (\Omega)\times {H}^{\ha}(\pO)\times {H}^{-\ha}(\pO)$
is a linear continuous operator, as requested.
We still have to check that the functions $(\tilde f_*,\Phi_*,\Psi_*)$, given by \eqref{f*}, \eqref{4.8Psi2T3}, 
\eqref{Phi*F}, 
satisfy equations \eqref{4.8Psi1} and \eqref{4.8Psi2}.
Indeed, $\Phi_*$ given by \eqref{Phi*eq} satisfies equation \eqref{Phi*F} and thus 
$\gamma^+W_\Delta (a\Phi_*)  = \gamma^+\F_\Delta$.
Since both $W_\Delta (a\Phi_*) $ and  $\F_\Delta$ are harmonic functions, this implies \eqref{WF}-\eqref{4.8Psi1c} 
and by \eqref{f*} also \eqref{4.8Psi1}.
Finally, \eqref{4.8Psi2T3} implies by \eqref{4.8Psi2TT} that \eqref{4.8Psi2T} is satisfied, and adding \eqref{6.32} to it 
leads to \eqref{4.8Psi2}.
}
\end{proof} 
\comment{
\begin{remark}\label{6.7}
The operator ${\cal C}_*$ from Lemma~\ref{L4} is not unique.
Indeed, let the triple
$(\tilde f_*,\Phi_*,\Psi_*)\in \s{H}^{-1} (\Omega)\times {H}^{\ha}(\pO)\times {H}^{-\ha}(\pO)$ 
be a solution of linear system \eqref{4.8Psi1}-\eqref{4.8Psi2} with $\F_1=0$ and $\F_2=0$. 
Then \eqref{2.6'genPhi} implies that $r_\Omega \tilde f_{*}=0$ in $\Omega$, i.e., 
$\tilde f_*\in{H}^{-1}_{\partial \Omega}\subset \s{H}^{-1} (\Omega)$.
Hence, \eqref{4.8Psi2T} reduces to
\begin{eqnarray}
\label{4.8Psi2T0}
0&=&T^+(\tilde f_{*},0)-\Psi_*
  \quad
\text{on }\partial\Omega.
\end{eqnarray}
By the first Green identity \eqref{Tgen}, this gives,
\begin{equation}
\label{Tgen*} 
0=\left\langle T^+(\tilde f_*,0)-\Psi_*, \gamma^{+}v \right\rangle _{\pO}
=\langle \tilde f_*-\gamma^*\Psi_*,v \rangle_\Omega \quad  \forall\ v\in H^1 (\Omega),
\end{equation}
which implies $\tilde f_*=\gamma^*\Psi_*$. 
Finally, \eqref{Phi*F} gives $\Phi_*=0$. 
Hence, any solution of linear system \eqref{4.8Psi1}-\eqref{4.8Psi2} with $\F_1=0$ and $\F_2=0$ has form
$(\tilde f_*,\Phi_*,\Psi_*)=(\gamma^*\Psi_*^0, 0, \Psi_*^0)$ for some $\Psi_*^0\in {H}^{-\ha}(\pO)$. 

On the other hand, the triple $(\tilde f_*,\Phi_*,\Psi_*)=(\gamma^*\Psi_*^0, 0, \Psi_*^0)$ with any 
$\Psi_*^0\in {H}^{-\ha}(\pO)$  solves \eqref{4.8Psi1}-\eqref{4.8Psi2} with $\F_1=0$ and $\F_2=0$ since then
${\cal P}\tilde f_* + V\Psi_*={\cal P}\gamma^*\Psi_*^0  -{\cal P}\gamma^*\Psi_*^0=0$
and 
\begin{multline*}
T^+(\tilde f_*+\mathring E\,{\mathcal R}_*\tilde f_*,{\cal P}\tilde f_*) -\frac{1}{2}\Psi_*+{\Wp}\Psi_*
=T^+(\tilde f_*+\mathring E\,{\mathcal R}_*\tilde f_*,{\cal P}\tilde f_*) -\Psi_* +T^+{V}\Psi_*\\
=(\gamma^{-1})^*(\tilde f_*+\mathring E\,{\mathcal R}_*\tilde f_*-\check A{\cal P}\tilde f_* )
-(\gamma^{-1})^*\gamma^*\Psi_* -(\gamma^{-1})^*(\mathring E\,{\mathcal R}_*\Psi_*
-\check A{\cal P}\Psi_*)\\
=(\gamma^{-1})^*(\tilde f_*-\gamma^*\Psi_*
+\mathring E\,{\mathcal R}_*(\tilde f_*-\gamma^*\Psi_*)-\check A{\cal P}(\tilde f_* -\gamma^*\Psi_*))=0
\end{multline*}
\end{remark}
}
\begin{theorem}\label{T3.521HF}
The cokernel of operator \eqref{mfN1} is spanned over the functional
\begin{align}\label{SSH10C}
g^{*1}:=((\gamma^{+})^*\partial_n a,1)^\top 
\end{align}
in  $\widetilde H^{-1}(\Omega )\times H^{\ha}(\pO)$, i.e.,
$
g^{*1}(\mathcal F_1,\mathcal F_2)
=\langle (\gamma^+\F_1)\partial_n a +\F_2,\gamma^+ u^0\rangle_{\partial\Omega},
$
where $u^0=1$.
\end{theorem}
\begin{proof} 

Let us consider the equation $\mathfrak N^{1}\mathcal U=(\mathcal F_1,\mathcal F_2)^\top$, i.e., the BDIE system (N1),
\begin{eqnarray}
\label{2.44HF}
 u +{\mathcal R}u  +W \varphi &=& \mathcal F_1  \quad \mbox{in}\quad   \Omega,
\\[1ex]
 \label{2.45HF}
 T^+{\mathcal R} u +\mathcal L^+\varphi
  &=&   \mathcal F_2  \quad   \mbox{on}\quad  \partial\Omega,
\end{eqnarray}
with arbitrary $(\mathcal F_1,\mathcal F_2)\in H^{1}(\Omega )\times H^{-\ha}(\pO)$, for $(u,\varphi)\in H^{1}(\Omega )\times H^{\ha}(\pO)$. 
By Lemma~\ref{L4},  the right hand side of the system has form \eqref{4.8Psi1**}-\eqref{4.8Psi2**}, i.e., 
system \eqref{2.44HF}-\eqref{2.45HF} reduces to
\begin{eqnarray}
\label{2.44HFr}
 u +{\mathcal R}u  +W (\varphi+\Phi_*) &=& {\cal P}\tilde f_{**}  \quad \mbox{in}\quad   \Omega,
\\[1ex]
 \label{2.45HFr}
 T^+{\mathcal R} u +\mathcal L^+(\varphi+\Phi_*) 
  &=&  T^+(\tilde f_{**}+\mathring E\,{\mathcal R}_*\tilde f_{**},{\cal P}\tilde f_{**})  \quad   \mbox{on}\quad  \partial\Omega,
\end{eqnarray}
where the couple $(\tilde f_{**}, \Phi_*)\in \s{H}^{-1} (\Omega)\times {H}^{\ha}(\pO)$ is given by \eqref{f**1}, \eqref{Phi*F**0}.
Up to the notations, system \eqref{2.44HFr}-\eqref{2.45HFr} is the same as in \eqref{FN1} with $\psi_0=0$. 
Then Theorems \ref{equivalenceN}(iii) and \ref{teoremP1inv} imply that BDIE system \eqref{2.44HFr}-\eqref{2.45HFr} and hence \eqref{2.44HF}-\eqref{2.45HF} is solvable if and only if
\begin{multline}\label{3.suf*}
\langle \tilde f_{**},u^0\rangle_\Omega 
=
\langle (\check{\Delta}(a\F_1)+\gamma^*(\F_2+(\gamma^+\F_1)\partial_n a),u^0\rangle_\Omega
=
\langle \nabla\cdot\mathring E \nabla (a\F_1) +\gamma^*(\F_2+(\gamma^+\F_1)\partial_n a),u^0\rangle_{\R^3}\\
=
-\langle \mathring E \nabla (a\F_1) ,\nabla u^0\rangle_{\R^3}
+\langle \F_2+(\gamma^+\F_1)\partial_n a,\gamma^+ u^0\rangle_{\partial\Omega} 
=\langle (\gamma^+\F_1)\partial_n a +\F_2,\gamma^+ u^0\rangle_{\partial\Omega}
=0,
\end{multline}
where we took into account that $\nabla u_0=0$ in $\R^3$.

Thus the functional $g^{*1}$ defined by \eqref{SSH10C} generates the necessary and sufficient solvability condition of equation $\mathfrak N^{1}\mathcal U=(\mathcal F_1,\mathcal F_2)^\top$. Hence $g^{*1}$ is a basis of the cokernel of  $\mathfrak N^{1}$.
\end{proof}
\begin{theorem}\label{T3.521HF2}
The cokernel of operator \eqref{mfN2} is spanned over the functional
\begin{align}\label{SSH10C2}
g^{*2}:=\left(\begin{array}{c}
-a\gamma^{+*}\left(\frac{1}{2}+\W'_\Delta\right)\mathcal{V}_{\Delta}^{-1}\gamma^+u^0
\\[1ex] 
-a\left(\frac{1}{2}-\W'_\Delta\right)\mathcal{V}_{\Delta}^{-1}\gamma^+u^0
\end{array}\right)
\end{align}
in  $\widetilde H^{-1}(\Omega )\times H^{-\ha}(\pO)$, i.e.,
$$
g^{*2}(\mathcal F_1,\mathcal F_2)
=\left\langle -a\gamma^{+*}\left(\frac{1}{2}+\W'_\Delta\right)\mathcal{V}_{\Delta}^{-1}\gamma^+u^0,\F_1\right\rangle_{\Omega}
+\left\langle -a\left(\frac{1}{2}-\W'_\Delta\right)\mathcal{V}_{\Delta}^{-1}\gamma^+u^0,\F_2 \right\rangle_{\partial\Omega},
$$
where $u^0(x)=1$.
\end{theorem}
\begin{proof} 
Let us consider the equation $\mathfrak N^{2}\mathcal U=(\mathcal F_1,\mathcal F_2)^\top$, i.e., the BDIE system (N2),
\begin{eqnarray}
\label{2.44HF2}
 u +{\mathcal R}u  +W \varphi &=& \mathcal F_1  \quad \mbox{in}\quad   \Omega,
\\[1ex]
 \label{2.45HF2}
 \frac{1}{2}\varphi+\gamma^+{\mathcal R} u   +{\mathcal W} \varphi
  &=&   \mathcal F_2  \quad   \mbox{on}\quad  \partial\Omega,
\end{eqnarray}
with arbitrary $(\mathcal F_1,\mathcal F_2)\in H^{1}(\Omega )\times H^{\ha}(\pO)$ for $(u,\varphi)\in H^{1}(\Omega )\times H^{\ha}(\pO)$. 

Introducing the new variable, $\varphi'=\varphi-(\F_2-\gamma^+\F_1)$, BDIE system \eqref{2.44HF2}-\eqref{2.45HF2} takes form
\begin{eqnarray}
\label{2.44HF2'}
 u +{\mathcal R}u  +W \varphi' &=& \mathcal F'_1  \quad \mbox{in}\quad   \Omega,
\\[1ex]
 \label{2.45HF2'}
 \frac{1}{2}\varphi'+\gamma^+{\mathcal R} u   +{\mathcal W} \varphi'
  &=&   \gamma^+\mathcal F_1'  \quad   \mbox{on}\quad  \partial\Omega,
\end{eqnarray}
where 
$$
\mathcal F_1'=\mathcal F_1-W(\F_2-\gamma^+\F_1)\in H^{1}(\Omega ).
$$
Let us recall that $\mathcal P=r_\Omega\mathbf P:\widetilde H^{s-2}(\Omega)\to H^{s}(\Omega)$ and then by Theorem~\ref{teoremP1inv}, the operator
$\mathcal P^{-1}=(r_\Omega\mathbf P)^{-1}:H^{s}(\Omega)\to \widetilde H^{s-2}(\Omega)$ is continuous for $\ha<s<\tha$.
Hence,  we can always represent $\mathcal F_1'=\mathcal P \tilde f_*$, with 
$$ \tilde f_*=[\Delta \mathring E (I- r_\Omega V_\Delta \mathcal{V}_{\Delta}^{-1}\gamma^+)
-\gamma^{+*}\mathcal{V}_{\Delta}^{-1}\gamma^+](a\mathcal F_1')\in \widetilde H^{-1}(\Omega ).
$$
For $\mathcal F_1'=\mathcal P \tilde f_*$, the right hand side of BDIE system \eqref{2.44HF2}-\eqref{2.45HF2} is the same as in \eqref{FN2} with $\tilde f=\tilde f_*$ and $\psi_0=0$. Then Theorem~\ref{equivalenceN}(iii) implies that BDIE system \eqref{2.44HF2'}-\eqref{2.45HF2'} is solvable if and only if
\begin{multline}\label{3.suf*'}
\langle \tilde f_*,u^0\rangle_\Omega 
=
\langle [\Delta \mathring E (I- r_\Omega V_\Delta \mathcal{V}_{\Delta}^{-1}\gamma^+)
-\gamma^{+*}\mathcal{V}_{\Delta}^{-1}\gamma^+](a\F'_1),u^0\rangle_{\R^3}\\
=
\langle \mathring E (I- r_\Omega V_\Delta \mathcal{V}_{\Delta}^{-1}\gamma^+)(a\F'_1),\Delta u^0\rangle_{\R^3}
-\langle \gamma^{+*}\mathcal{V}_{\Delta}^{-1}\gamma^+(a\F'_1),u^0\rangle_{\R^3}\\
=-\left\langle\gamma^+(a\F'_1),\mathcal{V}_{\Delta}^{-1}\gamma^+u^0\right\rangle_{\partial\Omega}\\
=-\left\langle\frac{1}{2}[\gamma^+(a\F_1)+(a\F_2)]-\W_\Delta[a(\F_2-\gamma^+\F_1)] ,\mathcal{V}_{\Delta}^{-1}\gamma^+u^0\right\rangle_{\partial\Omega}
\\
=\left\langle -a\gamma^{+*}\left(\frac{1}{2}+\W'_\Delta\right)\mathcal{V}_{\Delta}^{-1}\gamma^+u^0,\F_1\right\rangle_{\Omega}
+\left\langle -a\left(\frac{1}{2}-\W'_\Delta\right)\mathcal{V}_{\Delta}^{-1}\gamma^+u^0,\F_2 \right\rangle_{\partial\Omega}
=0.
\end{multline}
Thus the functional $g^{*2}$ defined by \eqref{SSH10C2} generates the necessary and sufficient solvability condition of equation $\mathfrak N^{2}\mathcal U=(\mathcal F_1,\mathcal F_2)^\top$. Hence $g^{*2}$ is a basis of the cokernel of  $\mathfrak N^{2}$.
\end{proof}

\subsection{Perturbed segregated BDIE systems for the Neumann problem}\label{PSBDIE}

Theorem \ref{equivalenceN} implies, that even when the solvability condition \eqref{3.suf} is satisfied, the solutions of both BDIE systems, (N1) and (N2), are not unique. By Theorem~\ref{invertibilityN}, in turn, the BDIE left hand side operators, $\mathfrak N^1$ and $\mathfrak N^2$, have non-zero kernels and thus are not invertible. 
To find a solution $(u,\varphi)$ from uniquely solvable BDIE systems with continuously invertible left hand side operators,
let us consider, following \cite{MikPert1999}, some BDIE systems obtained from  (N1) and (N2) by finite-dimensional operator perturbations. Note that other choices of the perturbing operators are also possible.

Below we use the notations $\mathcal U^N=(u,\varphi)^\top$ and $|\pO|:=\int_\pO dS$.

\subsubsection{Perturbation of BDIE system (N1)}\label{S6.3.1}

Let us introduce the perturbed counterparts of the BDIE system  (N1),
\begin{align}\label{N1-p}
\hat{\mathfrak N}^{1}\mathcal U^N=\mathcal F^{N1},
\end{align}
where
$$
\hat{\mathfrak N}^{1}:={\mathfrak N}^{1}+\mathring{\mathfrak N}^{1}
\ \text{and}\
\mathring{\mathfrak N}^1\mathcal U^N(y):=
g^0(\mathcal U^N)\,{\mathcal G}^1(y)=\frac{1}{|\pO|}\int_\pO \varphi(x)dS\ 
\left(\begin{array}{c}0\\ 1\end{array}\right) ,
$$
that is, 
\bes 
g^0(\mathcal U^N):=\frac{1}{|\pO|}\int_\pO \varphi(x)dS,
\quad
{\mathcal G}^1(y):=\left(\begin{array}{c}0\\ 1\end{array}\right).
\ees

For the functional $g^{*1}$ given by \eqref{SSH10C} in Theorem~\ref{T3.521HF}, 
$g^{*1}({\mathcal G}^1)=|\pO|$,
while
$g^0(\mathcal U^0)=1.$
Hence Theorem~\ref{L1} in Appendix, extracted from \cite{MikPert1999}, implies the following assertion. 
 \begin{theorem}
(i) The operator 
$
\hat{\mathfrak N}^{1}: H^1(\Omega )\times
H^{\frac{1}{2}}(\partial\Omega) \to H^1(\Omega )\times
H^{-\frac{1}{2}}(\partial\Omega)
$
is continuous and continuously invertable.

(ii)
If condition $g^{*1}(\mathcal F^{N1})=0$ (or condition \eqref{3.suf} for $\F^{N1}$ in form \eqref{FN1}) is satisfied, then the unique solution of  perturbed BDIDE system \eqref{N1-p} gives a solution of original BDIE system (D2) such that
\bes 
g^0(\mathcal U^N)=\frac{1}{|\pO|}\int_\pO \gamma^+ u dS=\frac{1}{|\pO|}\int_\pO \varphi dS=0.
\ees
\end{theorem}

\subsubsection{Perturbation of BDIE system (N2)}
The perturbation operators chosen below for BDIE system (N2) are slightly different from those, used in  \cite{MikPert1999} for the purely boundary integral equations, in \cite[Section 3]{MikNakJEM} for a united localised BDIE system and in \cite[Section 2]{MikMoh2012} for a united non-localised BDIE system. 

Let us introduce the perturbed counterparts of the BDIE system  (N1),
\begin{align}\label{N2-p}
\hat{\mathfrak N}^{2}\mathcal U^N=\mathcal F^{N2},
\end{align}
where
$$
\hat{\mathfrak N}^{2}:={\mathfrak N}^{2}+\mathring{\mathfrak N}^{2}
\ \text{and}\
\mathring{\mathfrak N}^2\mathcal U^N:=
g^0(\mathcal U^N)\,{\mathcal G}^2=\frac{1}{|\pO|}\int_\pO \varphi(x)dS 
\left(\begin{array}{c}a^{-1}(y)\\ \gamma^+a^{-1}(y)\end{array}\right)  ,
$$
that is, 
\bes 
g^0(\mathcal U^N):=\frac{1}{|\pO|}\int_\pO \varphi(x)dS,
\quad
{\mathcal G}^2:=\left(\begin{array}{c}a^{-1}(y)u^0(y)\\ \gamma^+[a^{-1}u^0](y)\end{array}\right).
\ees
For the functional $g^{*2}$ given by \eqref{SSH10C2} in Theorem~\ref{T3.521HF2}, since the operator $\mathcal{V}_{\Delta}^{-1}:H^\ha(\pO)\to H^{-\ha}(\pO)$ is positive definite and $u^0(x)=1$, there exists a positive constant $C$ such that
\begin{multline}
g^{*2}({\mathcal G}^2)=
\left\langle -a\gamma^{+*}\left(\frac{1}{2}+\W'_\Delta\right)\mathcal{V}_{\Delta}^{-1}\gamma^+u^0,a^{-1}u^0\right\rangle_{\Omega}
+\left\langle -a\left(\frac{1}{2}-\W'_\Delta\right)\mathcal{V}_{\Delta}^{-1}\gamma^+u^0,\gamma^+(a^{-1}u^0) \right\rangle_{\partial\Omega}\\
=-\left\langle \left(\frac{1}{2}+\W'_\Delta\right)\mathcal{V}_{\Delta}^{-1}\gamma^+u^0
+ \left(\frac{1}{2}-\W'_\Delta\right)\mathcal{V}_{\Delta}^{-1}\gamma^+u^0,\gamma^+u^0 \right\rangle_{\partial\Omega}
=-\left\langle\mathcal{V}_{\Delta}^{-1}\gamma^+u^0,\gamma^+u^0\right\rangle_{\partial\Omega}\\
\le -C\|\gamma^+u^0\|_{H^\ha(\pO)}^2
\le -C\|\gamma^+u^0\|_{L_2(\pO)}^2=-C|\pO|^2<0.\label{g0U0A2}
\end{multline}
Due to \eqref{g0U0A2} and  $g^0(\mathcal U^0)=1$, Theorem~\ref{L1} in Appendix implies the following assertion.
\begin{theorem}
(i) The operator 
$
\hat{\mathfrak N}^{2}: H^1(\Omega )\times
H^{\frac{1}{2}}(\partial\Omega) \to H^1(\Omega )\times
H^{\frac{1}{2}}(\partial\Omega)
$
is continuous and continuously invertable.

(ii)
If condition $g^{*2}(\mathcal F^{N2})=0$ (or condition \eqref{3.suf} for $\F^{N2}$ in form \eqref{FN1}) is satisfied, then the unique solution of  perturbed BDIDE system \eqref{N2-p} gives a solution of original BDIE system (N2) such that
\bes 
g^0(\mathcal U^N)=\frac{1}{|\pO|}\int_\pO \gamma^+ u dS=\frac{1}{|\pO|}\int_\pO \varphi dS=0.
\ees
\end{theorem}

\appendix \section*{APPENDIX}

\section{Function from $H^1(\Omega)$ with no classical or canonical conormal derivative}\label{Example}
For  functions from ${H}^{1}(\Om)$ the co--normal derivative $a\pa_n u$ on $\partial\Omega$ may not exist  in the classical (trace) or even canonical sense. In this section we consider an example of such function.

Let  $\Omega$ be a ball $B_{r_0}\subset\R^3$ of  some radius $r_0>0$ with the centre  at $x=0$. Let   $a=1$ and hence $A$ be the Laplace operator $\Delta$. Let us consider the function  
$$u(x)=(r_0^2-|x|^2)^{3/4},\quad x\in \Omega.$$

Evidently, this function is infinitely smooth in $\Omega$, vanishes on the boundary and its gradient 
\begin{align}\label{nu}
\nabla u(x)=-\frac{3}{2}x(r_0^2-|x|^2)^{-1/4}
\end{align}
belongs to $L_p(\Omega)$, $0< p<4$ and hence to $L_2(\Omega)$. This implies that $u$ belongs to the Sobolev space $W_p^1(\Omega$, $0< p<4$ and   thus $u\in H^1(\Omega)$.
For the classical conormal derivative we have, 
$$T_c^+u(x)=n(x)\cdot\lim_{|x|\to r_0} \nabla u(x)= -\infty,$$ 
which evidently means that it does not belong to any Sobolev space on the boundary.

 On the other hand, $u$ solves the Dirichlet problem
 \begin{align}
 \Delta u&=f\in H^{-1}(\Omega) \text{ in } \Omega, \\
 \gamma^+u&=0 \text{ on } \partial\Omega
 \end{align}
 with 
 $$
 f(x)=-\frac{9}{2}(r_0^2-|x|^2)^{-1/4}
 +\frac{3}{4}|x|^2(r_0^2-|x|^2)^{-5/4}\in H^{-1}(\Omega).
 $$

To define the canonical conormal derivative of $u$ according to Definitions \ref{Hst} and \ref{Tcandef}, the function $f$ should at least belong to $H^{-\ha}(\Omega)$. Let us prove that this is not the case. 
Indeed, if 
$f\in H^{-\ha}(\Omega)$, 
then the dual form 
$\langle f, \tilde g\rangle_\Omega$ 
should be bounded for any test function
$\tilde g\in \widetilde H^{\ha}(\Omega)$. 
Let us take  
$$
\tilde g(x)=\begin{cases}
(r_0^2-|x|^2)^{1/4}, & x\in\Omega\\
 0,& x\notin\Omega
 \end{cases}.
 $$ 
 Estimating the Sobolev-Slobodetski norm of this function one can prove that  
$\tilde g$ belongs to the space $\widetilde H^{s}(\Omega)$ for any $s<2/3$ and particularly to $\widetilde H^{\frac{1}{2}}(\Omega)$.
However
$$
f(x)\tilde g(x)=-\frac{9}{2}
 +\frac{3}{4}|x|^2(r_0^2-|x|^2)^{-1}\text{ in }\Omega
$$
and hence 
$
\langle f, \tilde g\rangle_\Omega=\int_\Omega f(x)\tilde g(x) dx
$
is not bounded. This implies that $f\notin H^{-\ha}(\Omega)$ and the canonical conormal derivative is also not defined.

To calculate the generalised co-normal derivative, one has to  extend the function $f\in H^{-1}(\Omega)$ to the function  
$\tilde f\in \widetilde H^{-1}(\Omega)$. As remarked in  \cite[Lemma 2.15]{MikJMAA2011} this is always possible due to the Hahn-Banach theorem, and an explicit extension is suggested in \cite[Theorem 2.16]{MikJMAA2011}, although the extension is not unique. Particularly,  one can assign $\tilde f= \check{A}u$, i.e.,  by  \eqref{Ltil},
\begin{align}\label{Ltil-e}
  \langle \tilde f,v \rangle_\Omega 
    =-\int_{\Omega} \nabla u(x)\cdot\nabla v(x)dx
        =-\int_{\Omega} \nabla u(x)\cdot\nabla v(x)dx
        =\langle \nabla\cdot\mathring E \nabla u,v\rangle_\Omega, \quad \forall v\in
    {H}^1(\Omega),
\end{align}
where $\nabla u$ is given by \eqref{nu}. 
Then \eqref{Tgend} implies that the generalised conormal derivative, $T^+(\tilde f,u)$, is well defined on $\partial\Omega$ and is zero. 
Different extensions of $f$ to $\tilde f$ will lead to different conormal derivatives, and moreover, any distribution from $H^{-\ha}(\pO)$ can be nominated as conormal derivative by appropriate choice of extension $\tilde f$, cf. \cite[Section 2.2, item 4]{Agranovich2003RMS}, \cite[Eq. (3.13)]{MikJMAA2011}, \cite[Eq. (5.10)]{MikJMAA2013}.

\section{Approximation of generalised conormal derivatives by classical ones}

\begin{theorem}\label{Tseq}
Let $u\in {H}^{1}(\Omega)$, $Au=r_\Omega\tilde f$ in $\Omega$ for some $\tilde f\in
\s{H}^{-1}(\Omega)$, and  $\{\tilde f_k\}\in\mathcal D(\Omega)$ be a sequence such that 
$\|\tilde f-\tilde f_k\|_{\widetilde H^{-1}(\Omega)}\to 0$ as $k\to\infty$. 
Then there exists a sequence  $\{u_k\}\in\mathcal D(\overline\Omega)$ such that 
$Au_k=r_\Omega\tilde f_k$ and $\|u-u_k\|_{H^{1}(\Omega)}\to 0$  as $k\to \infty$. Moreover, $\|T^+(u_k)-T^+(\tilde{f},u)\|_{H^{-\ha}(\partial\Omega)}\to 0$   as $k\to \infty$.
\end{theorem}
\begin{proof}
Let us consider the Dirichlet problem
\begin{eqnarray}
\label{2.6t} && A\,u_k= \tilde f_k  \;\;\; \mbox{\rm in}\;\;\;\; \Omega,
\\ 
 \label{2.7t} &&  \gamma^+u_k=\varphi_k  \;\;\; \mbox{\rm on}\;\;\;\; \pO,
\end{eqnarray}
where $ \{\varphi_k\} \in \mathcal D(\pO)$ is a sequence converging to $\gamma^+ u$ in $H^\ha(\pO)$. 
By Theorem~\ref{Rem1},  the unique solution of problem \eqref{2.6t}-\eqref{2.7t} in $H^{1}(\Omega)$ is  
$u_k=(\mathcal A^D)^{-1}(\tilde f_k,\varphi_k)^\top$, where 
$(\mathcal A^D)^{-1}:H^{-1}(\Omega)\times H^{\ha}(\pO)\to H^{1}(\Omega)$ is a continuous operator. 
Hence the functions $u_k$ converge to $u$ in $H^{1}(\Omega)$ as $k\to\infty$. 
Due to infinite smoothness of the data $(\tilde f_k,\varphi_k)$ and the boundary $\pO$, the solution $u_k$ belongs to $\mathcal D(\overline\Omega)$ implying that its classical conormal derivative $T^+u_k$ is well defined. 
Since $\tilde A u_k=\tilde f_k\in\mathcal D(\Omega)\in L_2(\Omega)$, the canonical conormal derivative is also  well defined and equals to the classical one. 
Then subtracting \eqref{Tcandef} for $u_k$ from \eqref{Tgend}, we obtain,
$$
\left\langle
 T^+(\tilde f,u)-T^+u_k\,,\, w\right\rangle _{\pO}=
 \langle \tilde f-\tilde f_k,\gamma^{-1}w \rangle_\Omega + \E(u-u_k,\gamma^{-1}w)\quad 
 \forall\ w\in H^{1/2} (\partial\Omega).
$$
Then 
\be \label{T-Tk}
\|T^+(\tilde f,u)-T^+u_k\|_{H^{-\ha}(\partial\Omega)}\le 
 C\left(\|\tilde f-\tilde f_k\|_{\widetilde H^{-1}(\Omega)} + \|\nabla (u-u_k)\|_{L_2(\Omega)}\right)
\ee
for some positive $C$.
Since the right hand side  of \eqref{T-Tk} tends to zero as $k\to\infty$, so does also the left hand side.
\end{proof}

\section{Properties of the surface and volume potentials}\label{SP}

The mapping and jump properties of the potentials of type \eqref{4.9bP}-\eqref{4.9P}, 
\eqref{3.6}-\eqref{3.7} and the corresponding boundary integral and
pseudodifferential operators in the H\"{o}lder ($C^{k+\alpha}$),
Bessel potential ($H^s_p$) and Besov ($B^s_{p,q}$) spaces  are
well studied nowadays for the constant coefficient, $a=const$,
(see, e.g., a list of references in \cite{CMN-1, Hsiao-Wendland2008}). Employing relations \eqref{VWab1}-\eqref{VWab4}, some of the
properties were extended in \cite{CMN-1, CMN-2} to the case of
variable positive coefficient $a\in C^\infty(\R)$, and several of those results are provided here for convenience
(without proofs).

\begin{theorem}\label{T3.1P}
Let $\Omega$ be a bounded open three--dimensional region of $\R^3$
with a simply connected, closed, infinitely smooth boundary. The
following operators are continuous
%
\begin{eqnarray} {\P}  &:& \widetilde{H}^{s}(\Omega) \to
H^{s+2}( {\Omega}), \quad s\in \R,\label{T3.1P1}\\
 &:& {H}^{s}(\Omega) \to H^{s+2}( {\Omega}), \quad
s>-\ha; \label{T3.1P2}\\
{\cal R},{\cal R}_*  &:& \widetilde{H}^{s}(\Omega) \to
H^{s+1}( {\Omega}), \quad s\in \R, \label{T3.1P3}\\
&:& {H}^{s}(\Omega) \to H^{s+1}( {\Omega}), \quad
s>-\ha;\label{T3.1P4}\\
 \gamma^+{\P}  &:& \widetilde{H}^{s}(\Omega) \to
H^{s+\frac{3}{2}}( {\pO}), \quad s>-\frac{3}{2},\label{T3.1P1+}\\
 &:& {H}^{s}(\Omega) \to H^{s+\frac{3}{2}}( {\pO}), \quad
s>-\ha; \label{T3.1P2+}\\
\gamma^+{\cal R}  &:& \widetilde{H}^{s}(\Omega) \to
H^{s+\ha}( {\pO}), \quad s>-\ha,\label{T3.1P3+}\\
&:& {H}^{s}(\Omega) \to H^{s+\ha}( {\pO}), \quad
s>-\ha;\label{T3.1P4+}\\
T^+{\P}  &:& \widetilde{H}^{s}(\Omega) \to
H^{s+\frac{1}{2}}( {\pO}), \quad s>-\frac{1}{2},\label{T3.1P1T+}\\
 &:& {H}^{s}(\Omega) \to H^{s+\frac{1}{2}}( {\pO}), \quad
s>-\ha; \label{T3.1P2T+}\\
T^+{\cal R}  &:& \widetilde{H}^{s}(\Omega) \to
H^{s-\ha}( {\pO}), \quad s>\ha,\label{T3.1P3T+}\\
&:& {H}^{s}(\Omega) \to H^{s-\ha}( {\pO}), \quad
s>\ha.\label{T3.1P4T+}
\end{eqnarray}
\end{theorem}
\begin{corollary}
\label{T3.1P0s}  The following operators are continuous,
 \begin{eqnarray}
  {\cal P}&:& \s{H}^{s}(\Omega)\to
H^{s+2,-\ha}(\Omega;L),\quad s\ge
-\ha,\label{T3.1P1ha}\\
 &:&
{H}^{s}(\Omega)\to H^{s+2,-\ha}(\Omega;L),\quad
s>-\ha;\label{T3.1P2ha}
\\
{\cal R}&:& H^{s}(\Omega) \to H^{s+1,-\ha}(\Omega;L), \quad s>\ha.
\label{T3.1P3ha}
\end{eqnarray}
\end{corollary}
\begin{proof}  Continuity of operators \eqref{T3.1P1}, \eqref{T3.1P2}
and \eqref{T3.1P4} imply continuity of operator \eqref{T3.1P1ha} for
$s>-\ha$ as well as \eqref{T3.1P2ha} and \eqref{T3.1P3ha}. 

\comment%
{
Let us prove \eqref{T3.1P1ham}. For $g\in
\s{H}^{-\ha}_\#(\Omega)\subset\s{H}^{s}(\Omega)$, $s< -\ha$, we
have, ${\cal P}\,g\in H^{s+2}(\Omega)$ due to \eqref{T3.1P1}, and
\begin{eqnarray}
\label{3.d1B} \Delta {\cal P}\,g\, =\Delta\left[\frac{1}{a }\;{\cal
P}_\Delta\,g\, \right]=&&\nonumber\\
\frac{1}{a }g\, +2\sum_{j=1}^3\partial_j\left[\frac{1}{a
}\right]\;\partial_j\left[{\cal P}_\Delta\,g\, \right]+
\left[\Delta\frac{1}{a }\right]\;{\cal P}_\Delta\,g&&
\mbox{in}\quad \R^3,\label{DPgB}
\end{eqnarray}
where ${\cal P}_\Delta:={\cal P}|_{a=1}$, and we taken into
account that $\Delta{\cal P}_\Delta\,g=g$. 
The first term in \eqref{DPgB}
belongs to $\s{H}^{-\ha}_\#(\Omega)$, while, since $a\in
C^\infty(\bar{\Omega})$, $a>0$, the sum of the second
and the third term belongs to ${H}^{s+1}(\Omega)$ for any $s<-\ha$
and can be extended by zero to $\s{H}^0(\Omega)\subset
\s{H}^{-\ha}(\Omega)$, which completes the proof of continuity for
operator \eqref{T3.1P1ham} since $H_\#^{t,-\ha} (\Omega,\Gamma;L)=
H_\#^{t,-\ha} (\Omega,\Gamma;\Delta)$, $t\ge \ha$.
} 

Let us prove \eqref{T3.1P1ha} for $s=-\ha$. For $g\in
\s{H}^{-\ha}(\Omega)$, we
have, ${\cal P}\,g\in H^\tha(\Omega)$ due to \eqref{T3.1P1}, while
\begin{eqnarray}
\label{3.d1B} \Delta {\cal P}\,g\, =\Delta\left[\frac{1}{a }\;{\cal
P}_\Delta\,g\, \right]=&&\nonumber\\
\frac{1}{a }g\, +2\sum_{j=1}^3\partial_j\left[\frac{1}{a
}\right]\;\partial_j\left[{\cal P}_\Delta\,g\, \right]+
\left[\Delta\frac{1}{a }\right]\;{\cal P}_\Delta\,g&&
\mbox{in}\quad \R^3,\label{DPgB}
\end{eqnarray}
where ${\cal P}_\Delta:={\cal P}|_{a=1}$, and we taken into
account that $\Delta{\cal P}_\Delta\,g=g$. 
The first term in \eqref{DPgB}
belongs to $\s{H}^{-\ha}(\Omega)$, while, since $a\in
C^\infty(\bar{\Omega})$, $a>0$, the sum of the second
and the third term belongs to ${H}^{\ha}(\Omega)$
and can be extended by zero to $\s{H}^0(\Omega)\subset
\s{H}^{-\ha}(\Omega)$, which completes the proof of continuity for
operator \eqref{T3.1P1ha} for $s=-\ha$.
\end{proof}

\begin{theorem}\label{T3.1s0}
The following operators are continuous,
 \begin{eqnarray}
V  &:& H^{s-\frac{3}{2}}(\pO) \to H^{s}( {\Omega}),\quad s\in \R,\label{VHs1}\\
   &:& H^{s-\frac{3}{2}}(\pO) \to H^{s,-\ha}( {\Omega;L})\label{VHs1Ga} ,\quad s> \ha;\\
W  &:&  H^{s-\ha}(\pO)\to H^{s}(\Omega),\quad s\in \R,\label{WHs1}\\
   &:&  H^{s-\ha}(\pO)\to  H^{s,-\ha}( {\Omega;L}) ,\quad s> \ha.\label{WHs1Ga}
\end{eqnarray}
\end{theorem}
\comment{
\begin{proof} 
\begin{eqnarray}\label{T3.1a}
V\Psi(y)=\frac{1}{a(y)}V_\Delta \Psi(y),&&
 V_\Delta \Psi(y):= \int\limits_{S}{P_\Delta}(x,y)\,\Psi(x)\,dx\\
\label{T3.1b} W\Phi(y)=\frac{1}{a(y)}W_\Delta[a\Phi](y),&&
 W_\Delta[a\Phi](y):=  \int\limits_{S}
\,\frac{\partial{P_\Delta}(x,y)}{\partial
n(x)}\,a(x){\Phi}(x)\,dx,\qquad\quad
\end{eqnarray}
where ${P_\Delta}(x,y):=-(4\pi)^{-1}\,|x-y|^{-1}$ is the
fundamental solution to the Laplace equation.

 This is well known that  the operators
\begin{eqnarray}\label{T3.1c}
V_\Delta   &:& H^{s-\frac{3}{2}}(\pO) \to H^{s}( {\Omega}) ,
\\ 
\label{T3.1d} W_\Delta  &:&  H^{s-\ha}(\pO)\to H^{s}(\Omega)
\end{eqnarray}
are continuous for any $s\in\R$ (see e.g. the above references).
Since $a(x)\not=0$ and $a\in C^\infty(\R)$, equalities
\eqref{T3.1a}, \eqref{T3.1b} imply  the similar properties,
\eqref{VHs1}, \eqref{WHs1}, for the operators $V$ and $W$.

On the other hand,
\begin{eqnarray*}
[\Delta V\Psi](y)=
\left[\Delta\frac{1}{a(y)}\right]V_\Delta \Psi(y)+
 \sum_{i=1}^3\frac{\partial }{\partial
 y_i}\left[\frac{1}{a(y)}\right]
 \frac{\partial V_\Delta \Psi(y)}{\partial y_i},&&
 \\
  {[}\Delta W\Phi](y)=
\left[\Delta\frac{1}{a(y)}\right]V_\Delta [a\Phi](y)+
 \sum_{i=1}^3\frac{\partial }{\partial
 y_i}\left[\frac{1}{a(y)}\right]
 \frac{\partial W_\Delta[a\Phi](y)}{\partial y_i},&&
\end{eqnarray*}
since $\Delta V_\Delta  \Psi(y)=\Delta W_\Delta
 [a\Phi](y)=0$ for $y\in \Omega$.

Due to the continuity of  operators \eqref{T3.1c}, \eqref{T3.1d},
this implies the operators
\begin{eqnarray*}
\Delta V &:& H^{s-\frac{3}{2}}(\pO) \to H^{s-1}( {\Omega}) ,
\\ 
\Delta W  &:&  H^{s-\ha}(\pO)\to H^{s-1}(\Omega)
\end{eqnarray*}
are continuous for $s\in\R$. Since $H^{s-1}(\Omega)\subset
\s{H}^{-\ha}(\Omega)$ for $s> \ha$, this implies \eqref{VHs1Ga},
\eqref{WHs1Ga} and completes the theorem.
\end{proof} 
} 
\begin{theorem}\label{T3.3}
Let $s\in \R$. The following pseudodifferential operators are
continuous
\begin{eqnarray*}
{\cal V}            &:& H^{s}(\pO)\to H^{s+1}(\pO),    \\
{\cal W},\;{\Wp} &:&  H^{s}(\pO)\to H^{s+1}(\pO), \\
{\cal L}^{+ }    &:& H^{s}(\pO)\to H^{s-1}(\pO).
\end{eqnarray*}
\end{theorem}

\begin{theorem}\label{T3.4}
Let $s\in \R$. The operators
\begin{eqnarray*}
r_{_{S_2}}\; {\cal V}\; : \;
H^{s}(\pO)\to H^{s}(\pO), && \\
r_{_{S_2}}\; {\cal W}\; : \;
H^{s}(\pO)\to H^{s}(\pO), && \\
r_{_{S_2}}\; {\Wp} \; :\;   H^{s}(\pO)\to H^{s}(\pO) &&
\end{eqnarray*}
are compact.
\end{theorem}
\begin{theorem}\label{T3.6}
The operator
\begin{equation*}
 {\cal V}  \;:\;
H^{s-1}(\pO)\to H^{s}(\pO)
\end{equation*}
is continuously invertible for all $s\in\R$.
\end{theorem}

\begin{corollary} The operators \label{B.3}
\begin{eqnarray}
\label{4.29}
 {\cal R}  &:& H^s(\Omega)\to H^{s}(\Omega),\quad s>-\ha,\\
\label{4.29d}
     &:& H^s(\Omega)\to H^{s,-\ha}(\Omega;L),\quad s>\ha,\\
\label{4.30}
\gamma^+ {\cal R}  &:& H^s(\Omega)\to H^{s-\ha}(\pO), \quad s>-\ha,\\
\label{4.31}
 T^+{\cal R} &:& H^s(\Omega)\to
 H^{s-\frac{3}{2}}(\pO), \quad s>\ha,
\end{eqnarray}
are compact for any infinitely smooth boundary curve $\pO$.
\end{corollary}
\begin{proof} Compactness of the operators \eqref{4.29}, \eqref{4.30}
and \eqref{4.31} follow from \eqref{T3.1P4}, \eqref{T3.1P4+}, and
\eqref{T3.1P4T+}, respectively, and the Rellich compact embedding
theorem. Then \eqref{4.29} and \eqref{T3.1P4} imply \eqref{4.29d}.
\end{proof} 

\section{Finite dimensional perturbation of operator equations}

Theorem~\ref{L1} below is implied by \cite[Lemma 2]{MikPert1999}  (see also  \cite[\S 21]{Vain-Tren1974}, \cite[Section 21.4]{Trenogin1980}, where the particular
case, 
$ h^* _i(\stackrel{\circ }{x}_j)=\mathring x^*_{i}( h _j)= \delta _{ij}$, 
has been considered). Another approach, although with hypotheses similar to the ones in Theorem~\ref{L1}, is presented in \cite[Lemma 4.8.24]{Hackbush1995}.

\begin{theorem}\label{L1} 
Let $B_1$ and $B_2$ be two Banach spaces.
Let
$\underline{A}:B_1\rightarrow B_2$ be a linear continuous Fredholm operator with zero index,
$\underline{A}^*:B_2^*\rightarrow B_1^*$ be the operator adjoined to it,
and
$\dim \ker \underline{A}=\dim \ker\underline{A}^*=n<\infty$,  
where
$\ker\underline{A}=\mathrm{span}\{\mathring{x}_i\}_{i=1}^n$ $\subset B_1$,
$\ker\underline{A}^*=\mathrm{span}\{\mathring x^*_{i}\}_{i=1}^n\subset B_2^*$. 
Let 
$$
\underline{A}_1x:=\sum_{i=1}^k h _i h^* _i(x),
$$
where
$ h^* _i$, $ h _i \, (i=1,..., n)$
are  elements from $B^*_1$ and $B_2$, respectively, such that
\begin{equation}\label{2.3}
\det [ h^* _i
(\mathring{x}_j)]\not= 0,
\qquad
\det [\mathring x^*_{i}(h _j)]\not= 0 \qquad i,j=1,..., n.
\end{equation}
Then:

(i) the operator $\underline{A}-\underline{A}_1:B_1\rightarrow B_2$ is continuous and continuously invertable;

(ii) if $y\in  B_2$  satisfies the solvability  conditions, 
\begin{equation}
\label{1.4}
\mathring x^*_{i}(y)=0, \qquad i=1,..., n,
\end{equation}
of equation 
\begin{equation}
\label{1.1C}
\underline{A}x=y,
\end{equation} 
then the unique solution $x$ of equation 
\begin{equation}\label{2.1C}
(\underline{A}-\underline{A}_1)x=y, 
\end{equation}  is a
solution of equation (\ref{1.1C}) such that\begin{equation}\label{2.4} h^*
_i(x) = 0\qquad (i=1,..., k).\end{equation}

(iii) Vice versa, if $x$ is a solution of equation (\ref{2.1C}) satisfying
conditions (\ref{2.4}), then conditions (\ref{1.4})
are satisfied for the right-hand side $y$ of equation (\ref{2.1C})
and $x$ is a solution of equation (\ref{1.1C}) with the same right-hand
side $y$.
\end{theorem}

Note that more results about finite-dimensional operator perturbations are available in \cite{MikPert1999}.

\section*{Concluding remarks}

The Dirichlet and Neumann problems for a variable--coefficient PDE with  general
right-hand side functions from ${H}^{-1}(\Omega)$ and $\widetilde{H}^{-1}(\Omega)$, respectively, were equivalently reduced to two direct segregated
boundary-domain integral equation systems, for each of the BVPs. 
This involved systematic use of the generalised co-normal derivatives without assumption that thy reduce to classical or canonical co-normal derivatives. 
The operators associated with
the left-hand sides of all the BDIE systems were analysed  in corresponding
Sobolev spaces. It was shown that the operators of the BDIE systems for the Dirichlet problem are continuous and continuously invertible. For the Neumann problem the BDIE system operators are continuous but only Fredholm with zero index,  their kernels and co-kernels were analysed, and appropriate finite-dimensional perturbations were constructed to make the perturbed operators invertible and provide a solution of the original BDIE systems and the Neumann problem. A further analysis of spectral properties of the two second kind
equations obtained in the paper is needed to decide whether the
resolvent theory and the Neumann series method (cf. \cite{MikMatSb1983,
SW2001} and references therein) are efficient for solving the equations.

The same approach can be used to extend, to the general PDE right hand sides, the BDIE systems for the mixed problems, 
unbounded domains, BDIEs of more general scalar PDEs and the systems of PDEs, as well as to the united and localised BDIEs, for which the analysis is now available for the right hand sides only from $L^2(\Omega)$, see \cite{CMN-1}%
\nocite{CMN-LocJIEA2009, CMN-NMPDE-crack, CMN-MDEMP2011, CMN-Ext-AA2013, CMN-IEOT2013}%
--\cite{CMN-SysDir2015+}, \cite{MikMMAS2006}, \cite{AyeleMik-EMJ2011}, \cite{DufMikIMSE2015}, \cite{MikPorIMSE2015}, \cite{MikPorUKBIM2015}.
The conditions on smoothness of the variable coefficients and the boundary can
be also essentially relaxed.
\\

\noindent
{\bf Acknowledgement}\\
This research was supported by the  grants
EP/H020497/1: "Mathematical Analysis of Localized Boundary-Domain Integral Equations
 for Variable-Coefficient Boundary Value Problems" 
and
EP/M013545/1: "Mathematical Analysis of Boundary-Domain Integral Equations for Nonlinear PDEs"  
 from the EPSRC, UK.



\end{document}